\newif\ifjournal

\journalfalse

\ifjournal

\else

\documentclass[a4paper,11pt]{amsart}

\fi

\usepackage{amsmath, amssymb}

\usepackage{graphics}
\usepackage{mathrsfs}
\usepackage[utf8]{inputenc}
\usepackage[english]{babel}
\usepackage[utf8]{inputenc}
\usepackage[cmtip,all]{xy}
\usepackage{graphicx}
\usepackage{etoolbox}
\usepackage{a4wide}
\usepackage{float}
\usepackage{verbatim}
\usepackage{enumerate}
\usepackage[pdfborder={1 1 0.3},bookmarks=false]{hyperref}

\usepackage{amsthm}
\newtheorem{theorem}{Theorem}[section]

\newtheorem{lemma}[theorem]{Lemma}
\newtheorem{proposition}[theorem]{Proposition}
\newtheorem{corollary}[theorem]{Corollary}

\theoremstyle{definition}

\newtheorem{definition}[theorem]{Definition}

\newcommand\frakfamily{\usefont{U}{yfrak}{m}{n}}
\DeclareTextFontCommand{\textfrak}{\frakfamily}

\newcommand{\frakp}{\mathfrak{p}}
\newcommand{\frakq}{\mathfrak{q}}

\DeclareMathOperator{\Aut}{Aut}

\DeclareMathOperator{\Hom}{Hom}
\DeclareMathOperator{\End}{End}

\DeclareMathOperator{\Gal}{Gal}
\DeclareMathOperator{\Spec}{Spec}

\DeclareMathOperator{\Ext}{Ext}
\DeclareMathOperator{\Ind}{Ind}

\DeclareMathOperator{\Gl}{GL}

\DeclareMathOperator{\Tatel}{T_{\ell}}

\newcommand\Z{\ensuremath{\mathbf{Z}}}
\newcommand\N{\ensuremath{\mathbf{N}}}
\newcommand\R{\ensuremath{\mathbf{R}}}

\newcommand\Q{\ensuremath{\mathbf{Q}}}
\newcommand\F{\ensuremath{\mathbf{F}}}

\newcommand{\calA}{\mathcal{A}}
\newcommand{\calB}{\mathcal{B}}

\newcommand{\calD}{\mathcal{D}}
\newcommand{\calE}{\mathcal{E}}
\newcommand{\calF}{\mathcal{F}}

\newcommand{\calJ}{\mathcal{J}}
\newcommand{\calK}{\mathcal{K}}
\newcommand{\calL}{\mathcal{L}}

\newcommand{\pibar}{{\bar{\pi}}}
\newcommand{\Alg}[1]{\overline{#1}}
\newcommand{\bad}{\pi}
\newcommand{\badrat}{p}
\newcommand{\badover}{\underline{\pi}}
\newcommand{\frakbad}{\frakp}
\newcommand{\good}{\ell}
\newcommand{\gooda}{\ell_{1}}
\newcommand{\goodb}{\ell_{2}}
\newcommand{\frakgood}{\underline{\good}}
\newcommand{\gsc}{J}

\newcommand{\fpbar}{\Alg{\F}_{p}}

\newcommand{\rhobar}{\overline{\rho}}
\newcommand{\shg}{\text{SH} _{16}}

\newcommand{\catb}{{\mathscr{C}_3}}
\newcommand{\catd}{{\mathscr{C}_2}}
\newcommand{\catf}{{\mathscr{C}_7}}
\newcommand{\cate}{{\mathscr{D}_7}}
\newcommand{\catc}{\mathscr{C}}
\newcommand{\catgen}{\mathscr{C}_\good(\bad,i)}

\newcommand{\maxextcatc}{T_\catc}
\newcommand{\maxextcatb}{T_\catb}
\newcommand{\maxextcatd}{T_\catd}

\newcommand{\locfield}{\calK}
\newcommand{\locfieldext}{\calL}
\newcommand{\locfieldalt}{\calF}

\DeclareMathOperator{\Tategood}{T_{\good}}
\newcommand{\Tate}[1]{T_{#1}}
\newcommand*{\longhookrightarrow}{\ensuremath{\lhook\joinrel\relbar\joinrel\rightarrow}}

\newcommand{\extt}[3]{\scalebox{0.9}{$\Ext^1_{#1}(#2,#3)$}}

\begin{document}
\ifjournal

\else

\title{Modularity of abelian varieties over $\Q$ with bad reduction in one prime only}
\author{Hendrik Verhoek}

\begin{abstract}
We show that certain abelian varieties over $\Q$ with bad reduction
at one prime only are modular
by using methods based on the tables of Odlyzko and class field theory.
\end{abstract}

%
%
\maketitle
\tableofcontents

\fi

\section{Introduction} \label{sec:intro}

In this article we consider abelian varieties over $\Q$ with bad reduction in one prime only.
Fontaine (\cite{Fontaine:1985}) and Abrashkin (\cite{Abrashkin:1987})
showed that abelian varieties over $\Q$  (and some other number fields with small discriminant) 
with good reduction everywhere do not exist.
Abelian varieties over $\Q$ with bad reduction in one prime only were then
considered in \cite{Schoof:2005} and \cite{BrumerKramer:2001}.
The following is a summary of results proven in
\cite{Schoof:2005}, \cite{Schoof:2009} and \cite{Schoof:2011}:

\begin{itemize}
\item For $\badrat \in \{2,3,5,7,13\}$, there do not exist non-zero semi-stable
abelian varieties over $\Q$ with good reduction outside $\badrat$.

\item For $\badrat \in \{2,3,5\}$, there do not exist non-zero abelian varieties over
$\Q$ with good reduction outside $\badrat$ that become semi-stable over an at most tamely ramified extension
at $p$ of $\Q$.

\item Every semi-stable abelian variety over $\Q$ with good reduction outside $11$ is isogenous to a product of copies
of $J_0(11)$.

\item Every semi-stable abelian variety over $\Q$ with good reduction outside $23$ is isogenous to a product of copies of
$J_0(23)$.
\end{itemize}

In any of these results, 
every such non-zero abelian variety is \emph{modular}:
Let $f$ be a newform of weight $2$ and level $N$.
Then one can associate to $f$ an abelian variety $A_f$ over $\Q$
and hence a representation $\rho_{A_f, \good} : G_\Q \rightarrow \Aut(\Tatel(A_f))$.
An abelian variety that is isogenous to a product of abelian varieties of the form
$A_f$ is called a \emph{a (classical) modular abelian variety}.

We want to extend these existence results by proving the modularity of 
any abelian variety $A$ over $\Q$ with good reduction outside a fixed prime $\badrat$
that satisfies some ramification condition on the Galois representation
$\rho_{A,\good} : G_\Q \rightarrow \Aut(\Tatel(A))$ at $\badrat$.
The ramification condition on $\rho_{A,\good}$ will replace the ramification condition of semi-stability at $\badrat$.

The above results indicate that we have
to consider abelian varieties that become semi-stable over an extension that is not necessarily at most tamely ramified
at either $2,3,5,7$ or $13$.
Even though such an extension may not be at most tamely ramified, 
the methods we use to prove modularity of certain abelian varieties over $\Q$
require that we do bound the ramification at $\badrat$ in some way.
The next definition provides a method to do this in a fairly canonical way using ramification groups with upper numbering
(see Section \ref{sec:ramificationgroups}):
 
\begin{definition}
Let $i \in \R_{\geq -1}$ and let $G_\locfield$ be the Galois group of an extension
$\locfield$ of $\Q_\badrat$ inside $\Alg{\Q}_\badrat$
such that $\locfield/\Q_\badrat$ is finitely ramified.
Let $I_\locfield$ be the inertia subgroup of $G_\locfield$.
A Galois representation $\rho$ of $G_\locfield$ is said to be \emph{ramified of level $i$} if
for all $u \in \R$ with $u > i$ the ramification group with upper numbering $I_\locfield^{u}$ is contained in $\ker(\rho)$.
\end{definition}

We define the level of ramification for an abelian variety over $K$ using the natural action
of the absolute Galois group $G_K$ on its $\good$-adic Tate module $T_{\good}(A)$:

\begin{definition} \label{def:Tategoodaction}
Let $i \in \R_{\geq -1}$. Let $A$ be an abelian variety over $\locfield$ and
let $\good$ be a rational prime different than the residue field characteristic of $\locfield$.
Then $A$ is \emph{ramified of level $i$} if $\rho_{A,\good} : G_\locfield \rightarrow \Aut(T_{\good}(A))$
is ramified of level $i$.
\end{definition}

Just as the conductor of an abelian variety $A$ does not depend on the choice of $\good$ in the definition of
the representation $\rho_{A,\good}$, 
so the level of ramification of $A$ does not depend on the choice of $\good$.
We will show this in Section \ref{sec:levelram}.

In Section \ref{sec:modforms}, 
we find all newforms $f$ with level $N$ a power of $\badrat$, 
such that the level of ramification of $A_f$ satisfies a certain bound.
This bound is such that we are still able to use the tables of Odlyzko \cite{Martinet:1982}.
For such newforms $f$, these tables then enable us to prove the modularity of an abelian variety over
$\Q$ with good reduction outside $\badrat$ and with level of ramification at $p$ at most 
the level of ramification of $A_f$ at $\badrat$.
These newforms turn out to have level $27,32$ or $49$.
Using this information, we then prove the
modularity of abelian varieties over $\Q$ with good reduction
outside either $2$, $3$ or $7$ that have the same level of ramification as the newforms of level
resp. $32$, $27$ or $49$:

\begin{theorem} \label{maintheorem}
Let $A$ be an abelian variety over $\Q$.
\begin{itemize}
\item If $A$ has good reduction outside $2$
and is ramified of level $\frac{3}{2}$ at $2$,
then under the assumption of the generalized Riemann hypothesis,
$A$ is isogenous to a product of $J_{0}(32)$.

\item If $A$ has good reduction outside $3$
and is ramified of level $\frac{1}{2}$ at $3$,
then $A$ is isogenous to a product of $J_{0}(27)$.

\item If $A$ has good reduction outside $7$
and is ramified of level $0$ at $7$,
then under the assumption of the generalized Riemann hypothesis,
$A$ is isogenous to a product of copies of $J_{0}(49)$ and
a product of copies of a certain $2$-dimensional modular abelian variety $B$.
\end{itemize}

\end{theorem}

We will show that these modularity results are optimal using methods 
based on the tables of Odlyzko and class field theory.
As corollaries of Theorem \ref{maintheorem}, we obtain:

\begin{theorem} \label{thm:conductorj027}
Every abelian variety over $\Q$ of conductor $27$ is isogenous to $J_{0}(27)$
\end{theorem}

\begin{theorem} \label{thm:conductorj049}
Every abelian variety over $\Q$ of conductor $49$ is isogenous to $J_{0}(49)$.
\end{theorem}

\section{Group schemes and level of ramification} \label{sec:grplevelram}

In order to possibly generalize our results to abelian varieties over other number fields than $\Q$,
we work in a more general setting than necessary
in Section \ref{sec:grplevelram}.

\subsection{Ramification groups with upper numbering} \label{sec:ramificationgroups}

Let $\locfieldext/\locfield$ be a non-trivial Galois extension of local fields.
For $i \geq -1$, let $I_{\locfieldext/\locfield}^{i}$ be the $i$-th ramification group with upper numbering
as defined in \cite[p. 61, p. 73 and p. 75 Remark 1]{Serre:1979}.
 
Let $u_{\locfieldext/\locfield} := \sup \{ i \in \R_{\geq -1} : G_{\locfieldext/\locfield}^{i} \not= 1\}$
resp. $i_{\locfieldext/\locfield} := \sup \{ i \in \R_{\geq -1} : (G_{\locfieldext/\locfield})_{i} \not= 1\}$
indicate the maximum value of the upper resp. lower numbering
of the non-trivial ramification groups that occur in $\locfieldext/\locfield$.
Then $\varphi_{\locfieldext/\locfield}(i_{\locfieldext/\locfield}) = u_{\locfieldext/\locfield}$.

\begin{definition} \label{def:levramlocalext}
The Galois extension $\locfieldext/\locfield$ is \emph{ramified of level $i$} if $u_{\locfieldext/\locfield} \leq i$.
\end{definition}
Ramified of level $-1$ simply means unramified,
and ramified of level $0$ simply means at most tamely ramified.

\begin{definition} \label{def:levramglobext}
Let $L$ be a Galois extension of number field $K$ and let $\bad$ be a prime ideal in $O_{K}$.
The extension $L/K$ is said to be \emph{ramified of level $i$ at $\bad$} if
for any prime ideal $\badover$ in $L$ above $\bad$ the extension
$L_{\badover} / K_{\bad}$ is ramified of level $i$ as in Definition \ref{def:levramlocalext}.
\end{definition}
Let $I_{L/K}(\bad)$ denote the ramification group of $L_{\badover}/K_{\bad}$,
which depends only up to conjugacy on the choice of $\badover$.
Furthermore, we let $u_{L/K}(\bad) := u_{L_{\badover}/K_{\bad}}$ and $i_{L/K}(\bad) := i_{L_{\badover}/K_{\bad}}$.
The numbers $u_{L/K}(\bad)$ and $i_{L/K}(\bad)$ do not dependent on the choice of $\badover$.

We recall a corollary of Herbrand's Theorem which says that ramification groups with
upper numbering behave well under quotients:

\begin{lemma} \label{thm:unumquotient}
Let $\locfieldext/\locfield$ be a finite Galois extension of local fields with Galois group $G$,
$v \in \R_{\geq -1}$ and $H$ be a normal subgroup of $G$. Then $G^{v}H/H = (G/H)^{v}$.
\end{lemma}
\begin{proof}
See \cite[Proposition 14, p. 74]{Serre:1979}.
\end{proof}

\begin{lemma} \label{lem:levramquotient}
Let $\locfieldext$ be a Galois extension of a local field $\locfield$ that is ramified of level $i$.
Then every subextension $\locfieldalt/\locfield$ is ramified of level $i$. 
\end{lemma}
\begin{proof}
First suppose that $\locfieldext/\locfield$ is finite and let $\locfieldalt/\locfield$ be a subextension.
Then apply Lemma \ref{thm:unumquotient}
with $H=\Gal(\locfieldext/\locfieldalt)$, $G=\Gal(\locfieldext/\locfield)$ and $v=i$.
If $\locfieldext/\locfield$ is infinite, the assertion follows by definition of the ramification groups with upper numbering
together with the finite case.
\end{proof}

\begin{proposition} \label{prop:levramextension}
Let $i \in \R_{\geq -1}$ and let $\locfieldext/\locfieldalt$ be finite Galois extensions
of a local field $\locfield$ such that $\locfieldalt/\locfield$ is finite and
ramified of level $i$, and such that $\locfieldext/\locfieldalt$ is ramified of level $\min(i,0)$.
Let $f$ be the inertial degree and $e$ be the ramification degree of $\locfieldalt/\locfield$.
Then $\locfieldext/\locfield$ is ramified of level $i$.
\end{proposition}
\begin{proof}
If $i \leq 0$, then the proposition is clear.
We assume that $i>0$.
Let $G=\Gal(\locfieldext/\locfield)$, $H=\Gal(\locfieldext/\locfieldalt)$
and let $\sigma \in G^j$ with $j > i$.
By Lemma \ref{thm:unumquotient} we have $\sigma \in H$.
Hence $\sigma = 1$ by hypothesis that $\locfieldext/\locfieldalt$ is ramified of level $\min(i,0)$.
\end{proof}

Now we will translate the information of the level of ramification into a bound on the discriminant.
Let $\calD_{\locfieldext/\locfield}$ denote the different of $\locfieldext/\locfield$
and $\Delta_{\locfieldext/\locfield}$ the discriminant of $\locfieldext/\locfield$.

\begin{lemma} \label{lem:localdiscbound}
If $i \geq 0$ and $\locfieldext/\locfield$ is ramified of level $i$, 
then $v_{\locfield}(\Delta_{\locfieldext/\locfield}) < [\locfieldext:\locfield] ( i+1)$. 
\end{lemma}
\begin{proof}
If $I_{\locfieldext/\locfield}^{u}$ is trivial for $u > i$, then $u_{\locfieldext/\locfield} \leq i$.
By \cite[Proposition 1.3]{Fontaine:1985}, we have
$$
v_{\locfield}(\calD_{\locfieldext/\locfield}) =  u_{\locfieldext/\locfield} + 1 - (i_{\locfieldext/\locfield} +1)/e .
$$
This gives the inequality $v_{\locfield} (\calD_{\locfieldext/\locfield}) < i+1$. 
The statement follows since $v_{\locfield} (\Delta_{\locfieldext/\locfield}) = [\locfieldext:\locfield] \cdot v_{\locfield} (\calD_{\locfieldext/\locfield})$.
\end{proof}

\begin{proposition} \label{prop:globalrootdisc}
Let $L/K$ be a Galois extension of number fields and $\badover$ a prime in $O_{L}$ above the prime $\bad$ in $K$.
If $i \geq 0$ and $L/K$ is ramified of level $i$ at $\bad$, then $v_{ K_{\bad}  }(\Delta_{L/K}) < [L:K] (i+1)$. 
\end{proposition}
\begin{proof}
This follows by Lemma \ref{lem:localdiscbound} and 
from the equality 
$\frac{v_{K_{\bad}}(\Delta_{L_{\badover} / K_{\bad}  } ) }{ [L_{\badover} : K_{\bad}  ] }=
\frac{v_{K_{\bad}}(\Delta_{L / K} ) }{ [L:K]}$.
\end{proof}

\subsection{Group schemes}
 
Let $O_K$ be the ring of integers of a number field $K$.
Let $\bad$ be a prime in $O_K$ and let $O_{K}[\frac{1}{\bad}]$ be the ring of
$\bad$-integers of $K$. Let $\good$ be a rational prime such that $\bad$ does not divide $\good$.
In the introduction we defined the level of ramification for abelian varieties.
In this section we introduce a similar notion for finite flat commutative group schemes over $O_{K}[\frac{1}{\bad}]$.
The $\good$-torsion points of abelian varieties with good reduction outside $\bad$
or, more precisely, the $\good$-torsion of N\'eron models of these abelian varieties,
are examples of such group schemes.
Let $\gsc$ be a finite flat commutative group scheme over $O_{K}[\frac{1}{\bad}]$ and let
$$
\rho_{\gsc}: G_{K} \longrightarrow \Aut(\gsc(\Alg{K}))
$$
be the representation given by the natural Galois action on $\gsc(\Alg{K})$.
For $\frakq$ a prime in $O_{K}$, let $G_{\frakq}  \subset G_{K}$ be a decomposition group of $\frakq$.
Denote as usual the restriction of $\rho_{\gsc}$ to $G_\frakq$ by $\rho_{\gsc} | G_\frakq$.

\begin{definition}
Let $\gsc$ be a finite flat commutative group scheme over
$O_{K}[\frac{1}{\bad}]$ and $\frakq$ be a prime in $O_{K}$.
Then $\gsc$ is \emph{ramified of level $i$ at $\frakq$} if $\rho_{\gsc} | G_\frakq$ is
ramified of level $i$.
\end{definition}

We are now able to define the following category:

\begin{definition}
Let $i \in \R_{\geq -1}$, let $\good$ be a rational prime
and let $\bad$ be a prime in $O_{K}$ such that $\bad \nmid \good$.
Define $\catgen$ to be the category whose objects are finite flat commutative
group schemes $\gsc$ over $O_{K}[\frac{1}{\bad}]$ of
$\good$-power order such that $\gsc$ is ramified of level $i$ at $\bad$,
and whose morphisms are group scheme morphisms over $O_{K}[\frac{1}{\bad}]$.
\end{definition}

For finite flat commutative group schemes $\gsc'$ and $\gsc''$ over $O_K[\frac{1}{\bad}]$, 
we denote the set of equivalence classes of commutative extensions of $\gsc'$ by $\gsc''$
by $\extt{O_K[\frac{1}{\bad}]}{\gsc''}{\gsc'}$.
By \cite[Corollary 17.5]{Oort:1966}, this set is actually an abelian group.

\begin{proposition} \label{prop:categoryprop}
Let $i \geq 0$.
Any extension of group schemes in $\catgen$ is again
a group scheme in $\catgen$.
Conversely, 
if $\gsc$ is a group scheme in $\catgen$, then any closed flat subgroup scheme of $\gsc$ or
quotient by a closed flat subgroup scheme of $\gsc$ is again in $\catgen$.
\end{proposition}
\begin{proof}
Let $\gsc'$ and $\gsc''$ be two group schemes in $\catgen$.
Consider an extension $\gsc$ of $\gsc'$ by $\gsc''$.
Then $\gsc$ has $\ell$-power order.
It remains to check that $\gsc$ is ramified of level $i$ at $\bad$.
For $u > i$, we let $\sigma$ be an element in $I_{K}(\bad)^{u}$.
By the property of the upper numbering being stable under quotients,
it follows by hypothesis that $\sigma$ acts trivially on $\gsc'$ and $\gsc''$.
We let $\sigma$ act on $x \in \gsc(\Alg{K})$. 
The image of $(\sigma -1)x$ under $\gsc(\Alg{K}) \rightarrow \gsc'(\Alg{K})$ is zero,
hence $y := (\sigma-1)x \in \gsc''(\Alg{K})$.
It follows that $(\sigma-1)^{2}x=(\sigma-1)y = 0$. 
The order of $\gsc$ is $\good^{e}$ for some $e \in \N$.
Since $(\sigma-1)^{2}=0$ and because the order of $\gsc$ is $\good^{e}$, 
the group scheme $\gsc$ is killed by $\good^e$
and we obtain that $\sigma^{\good^{e}}=(\sigma-1 + 1)^{\good^{e}}=1$.
Since $\bad \nmid \good$, this means $\sigma$ is in the tame part of $I_{K}(\bad)$.
Apply Proposition \ref{prop:levramextension} to deduce that $\gsc$ is ramified of level $i$
at $\bad$.

For the converse, 
consider an 
extension $\gsc$ of finite flat commutative group schemes $\gsc''$
by $\gsc'$ over $O_{K}[\frac{1}{\bad}]$.
Suppose $\gsc$ is a group scheme in $\catgen$.
Then $\gsc'$ and $\gsc''$ have $\good$-power order.
By Lemma \ref{lem:levramquotient},
both $\Q(\gsc')/\Q$ and $\Q(\gsc'')/\Q$ are Galois
extensions having level of ramification $i$.
Therefore $\gsc'$ and $\gsc''$ are group schemes in $\catgen$.
\end{proof}

Note that it is possible that for arbitrary commutative finite flat group schemes $\gsc'$ and $\gsc''$
such that $\rho_{\gsc'}$ and $\rho_{\gsc''}$ are unramified at $\bad$,
the $\Alg{\Q}$-points of an extension $\gsc$ of $\gsc'$ by $\gsc''$
generate an extension over $K$ that is tamely ramified at $\bad$.
An example is 
the extension $T[\good,\badrat]$ of $\Z/\good\Z$ by $\mu_\good$
defined below.

\begin{definition} \label{def:katzmazur}
Let $T[\good,\badrat]$ be the group scheme 
$
T[\good,\badrat] = \Spec(
\bigoplus_{i=0}^{\good-1} 
\Z[\frac{1}{\badrat}][X_{i}]/(X_{i}^{\good} - \badrat^{-i}))
$
with group law as defined in \cite[p. 251, Section 8.7]{KatzMazur:1985}).
\end{definition}
Next we recall the definition of the root discriminant of a number field:

\begin{definition}
The \emph{root discriminant of a number field $K$} is $\delta_{K} := \Delta_{K/\Q}^{\frac{1}{[K:\Q]}}$.
\end{definition}

We determine an upper bound of the root discriminant of the extension
$\Q(\gsc)/\Q$ for a simple group scheme $\gsc$ in the category $\catgen$,
or more generally, for a group scheme in $\catgen$ annihilated by $\good$.
Let $\frakgood$ be a prime ideal in $O_{K}$ above $\good$.
By \cite[Theorem 2.1]{Fontaine:1985}, if $e_{\frakgood}$ denotes the ramification index of the extension
$K/\Q$ at $\frakgood$,
then every group scheme $\gsc$ in $\catgen$ that is annihilated by $\good^{n}$
has level of ramification $(e_{\frakgood}( n + \frac{1}{\good-1}) -1)$ at $\frakgood$.
If $\gsc$ is a group scheme in $\catgen$, then $K(\gsc)/K$ is ramified of level $i$ at $\bad$
and ramified of level $(e_{\frakgood}( n + \frac{1}{\good-1}) -1)$ at every prime ideal $\frakgood$ above $\good$.

Let $T_{\catgen}$ be the compositum of 
all fields $K(\gsc)$, where the $\gsc$ are group schemes in $\catc$ that are annihilated by $\good$.
We call $T_{\catgen}$ the \emph{maximal $\good$-torsion extension} of $\catgen$.
The extension $T_{\catgen}$ is contained in the maximal extension of $K$ inside $\Alg{K}$
that is ramified of level $i$ at $\bad$
and ramified of level $(e_{\frakgood}( 1 + \frac{1}{\good-1}) -1)$ at every prime $\frakgood$ above $\good$.
Using these two levels of ramifications, we have the following bound on the root
discriminant of the maximal $\good$-torsion extension:

\begin{proposition} \label{prop:actionrootdiscb}
Let $\badrat$ be the residue characteristic of the prime ideal $\bad$.
Then the root discriminant of the 
maximal $\good$-torsion extension $L=T_{\catgen}$ of $\catgen$ satisfies the following inequality:
$$
\delta_{L} < \delta_{K} \cdot \badrat^{1+i} \cdot \good^{1 + \frac{1}{\good-1}}.
$$
\end{proposition}
\begin{proof}
Let $f_{K/\Q}$ be the inertial degree of $\bad$ in $K/\Q$.
Fontaine \cite[Corollary 3.3.2]{Fontaine:1985} shows that
$$
v_{\Q_{\good}}(\delta_{L}) < v_{\Q_{\good}} (\delta_{K}) + 1 + \frac{1}{\good-1} .
$$
It therefore suffices to show that
$v_{\Q_{\badrat}} (\delta_L) < v_{\Q_\badrat} ( \delta_K ) + (1+i)$.
Recall the well-known identity:
$$
\Delta_{L/\Q} = \Delta_{K/\Q}^{[L:K]} N_{K/\Q}(\Delta_{L/K}) .
$$
In terms of root discriminants this becomes
$\delta_{L} = \delta_{K} \cdot N_{K/\Q}(\Delta_{L/K})^{\frac{1}{[L:\Q]} }$.
By Proposition \ref{prop:globalrootdisc} we have $v_{ K_{\bad} }(\Delta_{L/K}) < [L:K](i+1)$,
from which we deduce that:
\begin{eqnarray*}
 v_{\Q_\badrat} ( N_{K/\Q}(\Delta_{L/K})  ) & < & v_{\Q_\badrat} ( N_{K/\Q}(\bad^{[L:K] (i+1)})) = 
v_{\Q_\badrat} ( \badrat^{ (f_{K/\Q}) [L:K] (i+1)}) \\
& = & f_{K/\Q} \cdot [L:K] (i+1) \leq [L:\Q] (i+1) . 
\end{eqnarray*}
Therefore we obtain the desired inequality $v_{\Q_\badrat} (  \delta_{L} ) < v_{\Q_\badrat} (\delta_{K} ) + (i+1)$.
\end{proof}

If the bound on the root discriminant of Proposition \ref{prop:actionrootdiscb}
is smaller than $22$ (or $42$ under the assumption of the generalized Riemann hypothesis),
the extension $L$ will be finite; cf. \cite{Martinet:1982}.

\section{Abelian varieties} \label{sec:levelram}

We will show that the level of ramification at $\bad$ of an abelian variety over $K$ 
does not depend on the choice of the prime $\good$ as long as $\bad \nmid \good$.
Then we relate abelian varieties with level of ramification $i$ with the category $\catgen$
defined in Section \ref{sec:grplevelram}.
Because the level of ramification only concerns 
Galois representations of the decomposition group (or ramification group) of $\bad$,
we reduce to the local case: we assume $\locfield$ to be a local field with uniformizer $\pi$ and residue 
field characteristic $\badrat$.
Let $A$ be an abelian variety over $\locfield$.
The absolute Galois group $G_{\locfield}$ acts on $\good$-adic Tate module $\Tategood(A)$ 
by the continuous representation:
$$
\rho_{A,\good} : G_{\locfield} \longrightarrow \Aut(\Tategood(A)) .
$$

\begin{lemma} \label{lem:independence_levram}
Let $H \subset I_\locfield$ be a pro-$p$-group and let $\gooda,\goodb$ be two
rational primes different from $\badrat$.
Then
$H \subset \ker(\rho_{A,\gooda})$
if and only if $H \subset \ker(\rho_{A,\goodb})$.
\end{lemma}
\begin{proof}
We may assume $\locfield=\locfield^{nr}$.
In this case there exists a minimal extension $\locfieldext/\locfield$
such that $A$ is semi-stable over $\locfieldext$ (see \cite[Expos\'e IX, p. 355]{Grothendieck:1972};
it is a corollary of the semi-stable reduction Theorem).
This minimal extension $\locfieldext$ has the property that for all $\sigma \in G_\locfieldext$
we have $(\sigma-1)^2=0$.
Assume that $H \subset \ker(\rho_{A,\gooda})$.
Then $H \subset \Gal(\Alg{\locfield}/\locfieldext)$ by minimality of the extension $\locfieldext/\locfield$. 
This means that $(\rho_{A,\goodb}(h)-1)^2 = 0$ for some $h \in H$.
Hence the image of $\sigma$ under $\rho_{A,\goodb}$ is contained in a pro-$\goodb$-group,
while $H$ is a pro-$p$ group.
This implies that $h \in \ker(\rho_{A,\goodb})$.
Hence $H \subset \ker(\rho_{A,\goodb})$.
\end{proof}

As a corollary we get:

\begin{theorem} \label{thm:independence_levram}
Let $\locfield$ be a local field, let $A$ be an abelian variety over $\locfield$
and $\gooda,\goodb$ be two rational primes different from $\badrat$.
Let $\rho_{A,\gooda} : G_{\locfield} \longrightarrow \Aut(\Tate{\gooda}(A)$
and $\rho_{A,\goodb} : G_{\locfield} \longrightarrow \Aut(\Tate{\goodb}(A))$
be the representations of $G_{\locfield}$ on the $\gooda$ resp. $\goodb$-adic Tate module.
Then for all $u \in \R_{\geq -1 }$ it holds that
$I_{\locfield}^{u} \subset \ker(\rho_{A,\gooda})$ if and only if 
$I_{\locfield}^{u} \subset \ker(\rho_{A,\goodb})$.
\end{theorem}
\begin{proof}
We may assume that $\locfield$ is equal to the maximal unramified extension $\locfield^{nr}$ of $\locfield$:
if the extension $\Alg{\locfield}^{\ker(\rho_{A,\good})}/\locfield$ is unramified, then $A$ has good
reduction and the result follows from \cite{SerreTate:1968}.
Hence $u \geq 0$.
Furthermore, we may assume that $u > 0$:
if $u=0$ and $I_{\locfield}^0 \subset \ker(\rho_{A,\gooda})$, then either $I_{\locfield}^0 \subset \ker(\rho_{A,\goodb})$, in which case we are done,
or this does not hold in which case we interchange $\gooda$ and $\goodb$.
If $u > 0$ we apply Lemma \ref{lem:independence_levram} with $H=I_\locfield^u$.
\end{proof}

Let $K$ be a number field and let $A$ be an abelian variety over $K$ having good reduction
outside $\bad$. If $\calA$ is the N{\'e}ron model of $A$, consider the $\good^{n}$-torsion
subgroup scheme ${\calA} [\good^{n}]$, which is a finite flat commutative group scheme of $\good$-power order over 
$O_{K}[\frac{1}{\bad}]$.
The group scheme ${\calA} [\good^{n}]$ is ramified of level $i$ for some $i \in \R_{\geq -1}$.
Hence, the scheme ${\calA} [\good^{n}]$ is an object of the category $\catgen$.
The extension $K({\calA} [\good^{n}])/K$ ramifies only at $\bad$ and primes lying over $\good$.
The next proposition relates the action of inertia on the finite group schemes
of $\good^{n}$-torsion points and the $\good$-adic Tate module:

\begin{proposition} \label{prop:linktatecat}
Let $A$ be an abelian variety over a number field $K$ with good reduction away from 
the prime $\bad$ and let $\good$ be any rational prime such that $\bad \nmid \good$.
If $i \geq -1$, then:
$$
A  \text{ is ramified of level i at }  \bad \iff \calA[\good^{n}] \in \catgen \, \text{ for all } n \in \N
$$
If $i > 0$, then:
$$
A \text{ is ramified of level i at } \bad \iff \calA[\good]  \text{ is ramified of level i at } \bad.
$$
\end{proposition}
\begin{proof}
The first implication is obvious.
For the second statement, let $I_\bad$ be the inertia group of any decomposition group of $\bad$.
Any element in the wild part of $I_\bad$ acts trivially on $\calA[\good]$
if and only if it acts trivially on $\calA[\good^\infty]$, because 
$\Gal(K(\calA[\good^\infty])  / K(\calA[\good])$ is a pro-$\good$-group.
\end{proof}

Besides the independence of $\good$, we would like that the level of ramification
is well-defined on isogeny classes of abelian varieties, so that it can be used in order
to study these classes.
This is indeed the case:

\begin{proposition}
Let $A$ and $A'$ be two isogenous abelian varieties over $\locfield$.
Then $A$ is ramified of level $i$ if and only if $A'$ is ramified of level $i$.
\end{proposition}
\begin{proof}
Let $\psi: A \rightarrow A'$ be an isogeny between $A$ and $A'$.
Consider $\rho_{A,\good} : G_{\locfield} \rightarrow \Aut(\Tategood(A))$
and $\rho_{A',\good} : G_{\locfield} \rightarrow \Aut(\Tategood(A'))$.
If any group $H \subset G_{\locfield}$ acts trivially on $\Tategood(A)$
then it also acts trivially on $\Tategood(A')$, because $\psi(\Tategood(A))$
is a finite index subgroup of $\Tategood(A')$.
In particular this holds for the ramification groups $I_{\locfield}^{i}$.
\end{proof}

Before we move on to the next section, we will cite a theorem that
allows us to prove that certain abelian varieties are modular:

\begin{theorem} \label{cat-thm:abvar_torsionfilter}
Let $A$ be an abelian variety over $\Q$ such that $A[\good^n]$ is an object
in $\catgen$ for all $n \in \N$.
Suppose that
\begin{itemize}
\item for all simple non-\'etale group schemes $T$ in $\catgen$ 
and all simple \'etale group schemes $E$ in $\catgen$,
the group $\extt{\catgen}{T}{E}$ is trivial.
\item the maximal abelian extension $R$ of 
$F$, where $F$ is the compositum of all $\Q(E)$ with $E$ running over all simple \'etale group
schemes $E$ in $\catgen$,
that is unramified outside $\badrat$ and at most tamely
ramified at primes over $\badrat$, is a finite cyclic extension $F$.
\end{itemize}
Then $A[\good]$ does not have subquotients that are \'etale or 
of multiplicative type.
\end{theorem}
\begin{proof}
See \cite{Verhoek:2009}.
\end{proof}

\section{Modular forms} \label{sec:modforms}

Let $f$ be a newform of weight $2$ and level $N$.
One can associate to $f$ an abelian variety $A_f$ over $\Q$
and a representation $\rho_{A_f}$.
The representation $\rho_{A_f}$ obtained like this 
is the case $k=2$ of the next theorem:

\begin{theorem} \label{thm:repr_to_newform}
Let $f \in S_k(\Gamma_1(N),\epsilon)$ be a newform, let $E$ be its coefficient field
and let $a_p$ denote its Hecke eigenvalues for $p \nmid N$. 
Let $\lambda$ be a prime ideal in $E$ with residue characteristic $\good$.
Then there exists an irreducible representation
$\rho_f : G_\Q \longrightarrow \Gl_2(E_\lambda)$
such that $\rho_f$ is unramified outside $\good N$.
\end{theorem}
\begin{proof}
See for example \cite[Theorem 2]{Carayol:1987}.
\end{proof}

We will apply this to the case that $N$ is a power of $\badrat$.
We want to consider the mod-$\good$ representation $\rhobar_f$ 
coming from the above $\good$-adic representation $\rho_f$.
The representation $\rhobar_f$ is well-defined up to semi-simplification.
Suppose $\rhobar_f$ is ramified of level $i$ at $\badrat$.
Then the finite flat group scheme $A_f[\good]$ over $\Z[\frac{1}{\badrat}]$
is an object in the category $\catgen$.
To prove that all abelian varieties over $\Q$ with good reduction outside $\badrat$
and with level of ramification $i$ at $\badrat$ are modular,
we try to find all simple objects in the category $\catgen$.
By Proposition \ref{prop:actionrootdiscb},
the root discriminant $\delta_L$ of the maximal $\good$-torsion extension $L$ of $\catgen$
satisfies $\delta_L < \good^{1+\frac{1}{\good-1}} \badrat^{u+1}$.
If the bound on the root discriminant $\delta_L$ is smaller than $42$,
we can use the tables in \cite{Martinet:1982} to get an upper bound
on the degree of $L$ and prove that $L$ is finite.

Since there are no abelian varieties over $\Q$ with good
reduction everywhere, we assume that $i \geq 0$.
For $\badrat \geq 13$ and $\good=2$, the bound $\good^{1+\frac{1}{\good-1}} \badrat^{u+1}$ is 
larger than $42$.
Therefore, we consider only the primes $2,3,5$ and $7$.
Let $\badrat$ be one of these primes
and let $\good = 2$ if $\badrat \not=2$, otherwise let $\good=3$.

What we do next is translate a bound on the level of ramification at $p$
of $\rhobar_{f}$ into a bound on the exponent $n$ of the level $\badrat^n$ of the newform $f$.
Before we make this translation though, we recall the definitions of the exponent of the Artin conductor and the Swan conductor
that we will use in making this translation.

Let $F$ be a field of characteristic different from $p$.
Let $\rho : G_{\Q_p} \rightarrow \Gl(V)$ be a continuous representation
of $G_{\Q_p}$ on a finite dimensional topological $F$-vector space $V$.
Then \emph{the exponent of the Artin conductor} is  
$$
a(\rho) = \dim(V) - \dim(V^{I_p}) + \int_1^{\infty} \dim(V/V^{G^u})  \mathrm{d} u .
$$

We introduce additional notation for the representation $\rho$:
If $K' = \Alg{\Q_p}^{\ker{\rho}}$, then we define 
$$
u(\rho) := u_{K'/K} \quad \text{and} \quad i(\rho) := i_{K'/K} .
$$
We note that for a character $\chi$, the exponent $a(\chi)$ is equal to $u(\chi)+1$.

\begin{theorem} \label{thm:carayol}
If $f \in S_k(\Gamma_1(N))$ is a newform, then for all $p|N$ and $p \nmid \ell$
the equality $v_p(a(\rho_{f,p})) = v_p(N)$ holds. 
\end{theorem}
\begin{proof}
See \cite[Theorem 4]{Carayol:1987}.
\end{proof}

\begin{lemma} \label{lem:newform_level_ram}
Let $n \in \N_{\geq 2}$ and let $f \in S_2(\Gamma_1(\badrat^n),\epsilon)$ be a newform.
If $\rho_{f,\badrat}$ is
\begin{enumerate}[(a)]
\item irreducible, then $u(\rho_{f,\badrat})=\frac{n}{2}-1$.
\item decomposable such that 
$$
\rho_{f,\badrat} \simeq
\left(\begin{array}{cc}
\epsilon \omega_\good \chi & 0 \\0 & \chi^{-1}
\end{array}\right) ,
$$
then $u(\rho_{f,\badrat})=n-\min(a(\chi), a(\epsilon \chi))-1$. 
\item reducible and indecomposable such that 
$$
\rho_{f,\badrat} \simeq
\left(\begin{array}{cc}
\omega_\good \chi & * \\0 & \chi
\end{array}\right) ,
$$
where $\chi$ is a character with $\chi^2 = \epsilon$,
then $u(\rho_{f,\badrat})= \max( a(\chi) -1, 0)$.
\end{enumerate}
\end{lemma}
\begin{proof}

Let $E_\lambda$ be the $\lambda$-adic completion of the number field $E$
of Theorem \ref{thm:repr_to_newform}.
Let $V$ be the $2$-dimensional $E_\lambda$-vector space on which $G=G_{\Q_p}$ acts.
By Theorem \ref{thm:carayol} we have that $n=a(\rho_{f,p})$. 

Part (a) follows because $V$ is irreducible, so $\dim(V/V^{G^i})$ is either $2$ or $0$
depending on whether $G^i$ acts trivially or not.

For part (b), note that $\omega_\good$ is unramified and that 
$n = a(\rho_{f,p})  = a(\chi^{-1}) + a(\epsilon \chi) = a(\chi) + a(\epsilon \chi)$.
We may suppose without loss of generality that $a(\chi) \geq a(\epsilon \chi)$.
Then $n-a(\epsilon \chi)-1= a(\chi)-1=u(\chi)=u(\rho_{f,p})$. 

To prove part (c), we suppose that the $*$ does not vanish when restricted to the inertia subgroup, otherwise
the analysis is the same as the one in part (b).
We remark that if $a(\chi) = 0$ or $a(\chi) = 1$, then $u(\rho_{f,p}) = 0$.
Therefore we suppose that $a(\chi) > 1$, so that the representation $\rho_{f,p}$ is wildly ramified.
Then $u(\rho_{f,p}) = u(\chi)=a(\chi)-1 > 0$.
\end{proof}

\begin{proposition} \label{prop:newform_level_ram_lowerbound}
Let $f$ be a newform of level $p^n$.
If $\rho_f$ is ramified of level $i$ at $p$, then $n \leq 2(i+1)$.
If $\rhobar_f$ is ramified of level $i \geq 0$ at $p$, then $n \leq 2(i+1)$.
If $\rhobar_f$ is unramified at $p$, then $n \leq 2$.
\end{proposition}
\begin{proof}
The first part follows from Lemma $\ref{lem:newform_level_ram}$.
For example for the semi-simple case, the maximum value of $\min(a(\chi), a(\epsilon \chi))$ is bounded by $n/2$. 
The second statement follows because if $\rhobar_f$ is ramified of level $i \geq 0$ at $p$,
then so is $\rho_f$.

Finally, if $\rhobar_f$ is unramified at $p$, then $\rho_f$ is at most
tamely ramified and the third statement then follows again from
Lemma $\ref{lem:newform_level_ram}$.
\end{proof}

\begin{table}[htb]
  \centering
  \begin{tabular}{@{} |c|c|c|c|c|c|c| @{}}
    \hline
    newform & level & dim & conductor nebentypus & rep & $u(\rho_{f,p})$\\ 
     & & & & & \\
    \hline
    16A & $16$ & $2$ & $16$ & dec. &  $3$ \\
    32A & $32$ & $1$ & $1$ & irr. & $3/2$ \\
    32B & $32$ & $4$ & $32$ & dec. &  $4$ \\
    32C & $32$ & $8$ & $32$ & dec. & $4$ \\
    64A & $64$ & $1$ & $1$ & irr. & $2$ \\
    27A &  $27$ & $1$ & $1$ & irr. &  $1/2$ \\
    27B & $27$ & $12$ & $27$ & dec. & $2$ \\
    81A & $81$ & $2$ & $1$ & irr. & $1$ \\
    81B & $81$ & $2$ & $9$ & irr. & $1$ \\
    81C & $81$ & $4$ & $9$ & irr. &  $1$ \\
    81D & $81$ & $12$ & $27$ & dec. & $2$ \\
    81E & $81$ & $144$ & $81$ & dec. &  $3$ \\
    25A & $25$ & $4$ & $25$ & dec. &  $1$ \\
    25B & $25$ & $8$ & $25$ & dec. &  $1$ \\
    49A & $49$ & $1$ & $1$ & irr. &  $0$ \\
    49B & $49$ & $2$ & $7$ & irr. &  $0$ \\
    49C & $49$ & $6$ & $49$ & dec. &  $1$ \\
    49D & $49$ & $12$ & $49$ & dec. & $1$ \\
    49E & $49$ & $48$ & $49$ & dec. & $1$ \\
    \hline
  \end{tabular}
\vskip 3pt
  \caption{The "first" occurring newforms of prime power level.}

  \label{tab:modformlevel}
\end{table}

Proposition \ref{prop:newform_level_ram_lowerbound} allows us
to determine the maximal exponent $n$ of the level $\badrat^n$ of newforms $f$ of weight $2$ for which we
hope to determine all abelian varieties over $\Q$ having the same ramification level at $\badrat$ as $A_f$.
For instance, we determine the maximal exponent in the case $p=3$:
For this, we require that the inequality $2^{1+\frac{1}{2-1}} 3^{i+1} < 42$ is satisfied, so $i < 1.140$. 
We then use Proposition \ref{prop:newform_level_ram_lowerbound} to obtain that $n$ is at most $4$.
Doing the same for the primes $\badrat=2,5$ and $7$, we see that the exponent $n$ is at
most $6,2$ or $2$ respectively: Thus we are interested in newforms of level $16,32,64,27,81,25$ and $49$. 

The article \cite{LoefflerWeinstein:2010} permits us to determine for each newform $f$ 
if the representation $\rho_{f,p}$ is irreducible (coming from a supercuspidal automorphic representation),
decomposable (coming from a principal series automorphic representation) or reducible and indecomposable
(coming from a special automorphic representation).
In Table \ref{tab:modformlevel} each row corresponds to a newform of weight $2$.
In the column named \emph{rep} we indicate the form of the representation $\rho_{f,p}$:
in all cases considered, the representation turns out to be either decomposable or irreducible.

After a straight-forward computation we conclude that
only the newforms $32A,27A,49A$ and $49B$ satisfy $\good^{1+\frac{1}{\good-1}} \badrat^{i+1} < 42$,
where $\badrat$ is the prime dividing the level and $\good$ is chosen as above. 
In the next chapters we look at the newforms $32A,27A,49A$ and $49B$ and we
show that indeed all abelian varieties over $\Q$, that have the same ramification level at $p$ as
these modular forms, are modular.

\section{Good reduction outside $2$} \label{sec:j032}

In this section we prove the following theorem:

\begin{theorem} \label{thm:mainj032}
Let $A$ be an abelian variety over $\Q$ with good reduction outside $2$
such that $A$ is ramified of level $\frac{3}{2}$ at $2$.
Then, under the assumption of the generalized Riemann hypothesis,
$A$ is isogenous to a product of copies of the elliptic curve $J_{0}(32)$.
\end{theorem}

We remark that we also proved the following weaker statement 
without the assumption of the generalized Riemann hypothesis:
any non-zero abelian variety $A$ over $\Q$ with good reduction outside $2$
such that $A_\locfieldext$, where $\locfieldext$ is the ray class field of conductor $(1+i)^3$ of $\Q_2(i)$,
has good reduction, is isogenous to a product of copies of the elliptic curve $J_{0}(32)$.

From Table \ref{tab:modformlevel} we see that
there is only one isogeny class of elliptic curves of conductor $32$.
We consider the elliptic curve $E : y^{2}=x^{3}-x$ over $\Q$ of conductor $32$.
The elliptic curve $E$ has additive reduction at $2$ and is supersingular at $3$.
We continue to work with the elliptic curve $E$ in this section and write
$\calE$ for the N\'eron model of $E$.
Since the curve has complex multiplication by $\Q(i)$,
the Galois representation $\rho_{E,3}$ is an induction of a character and $E$ has potential
good reduction.
By the criterion of N\'eron-Ogg-Shafarevich, it obtains good reduction over the field $\Q(\calE[3])$ of degree $16$ over $\Q$. 
Let $\overline{\rho_{E,3}}$ be the reduction of the representation $\rho_{E,3}$ modulo the prime $3$,
which is well-defined up to semi-simplification.
A calculation (see the Appendix) using Magma \cite{BosmaCannonPlayoust:1997}
shows that the image of the representation $\overline{\rho_{E,3}}$ is isomorphic to 
the group $\shg$, the semi-hedral group of order $16$ mentioned
above. 

\begin{lemma} \label{lem:end_gal_equivariant}
We have $\End_{G_\Q}(T_{3}(E)) \simeq \Z_{3}$.
In particular, $\End_{G_\Q}(E[3]_\Q)$
has order $3$.
\end{lemma}
\begin{proof}
Since $E$ is supersingular at $3$ and has good reduction at $3$,
it follows from \cite[Theorem 1.1]{Conrad:1997} that $\overline{\rho_{E,3}}$ is absolutely irreducible.
We deduce that also $\rho_{E,3}$ is absolutely irreducible.
Now apply Schur's Lemma, see for instance \cite[Paragraph 4, Corollary, p. 252]{Mazur:1997}.
\end{proof}

We will prove Theorem \ref{thm:mainj032} by working in a certain category of group schemes.
Let $\gsc$ be a finite flat commutative group scheme over $\Z[\frac{1}{2}]$.
Then $G_\Q$ acts on the $\Alg{\Q}$-points of $\gsc$ by
$\rho_{\gsc}: G_{\Q} \longrightarrow \Aut(\gsc(\Alg{\Q}))$.
We denote by $\Q(\gsc)$ the fixed field of the kernel of the representation $\rho_\gsc$.
To prove Theorem \ref{thm:mainj032}, we use the following category:

\begin{definition}
Let $\catd$ be the category whose objects are finite flat commutative
group schemes $\gsc$ over $\Z[\frac{1}{2}]$ of
$3$-power order such that $\gsc$ is ramified of level $\frac{3}{2}$ at $2$,
and whose morphisms are group scheme morphisms over $\Z[\frac{1}{2}]$.
\end{definition}

The category $\catd$ is closed under taking closed flat subgroup schemes, quotients and products.
It is even closed under taking extensions by Proposition \ref{prop:categoryprop}.

\subsection{Simple group schemes} \label{sec:simplej032}

From Table \ref{tab:modformlevel} we see that $J_{0}(32)$ is
ramified of level $\frac{3}{2}$ at $2$.
Hence the group schemes $\calE[3^{n}]$ are 
objects in the category $\catd$. 
The group scheme $\calE[3]$ helps us to prove that the maximal
$3$-torsion extension of $\catd$ is the number field $\Q(\calE[3],\sqrt[3]{2})$.
First we state the following lemma:

\begin{lemma} \label{lem:combinedconductors}
Let $i,j \in \N$, let $K$ be a number field and let $\frakp, \frakq$ be two primes
in $K$ with residue characteristics $p$ resp. $q$ such that $p \not= q$.
Suppose that the ray class fields of conductor $\frakp^{i} \frakq$ and
of conductor $\frakp \frakq^{j}$ are trivial.
Then for all $m \leq i$ and $n \leq j$
the ray class field of conductor $\frakp^{m} \frakq^{n}$ is trivial.
\end{lemma}
\begin{proof}
Suppose the ray class field $\locfieldext$ of conductor 
$c=\frakp^{i} \frakq^{j}$ is non-trivial.
Look at the $p$-part of the extension $L/K$.
Since $L/K$ is abelian the $p$-part of $L/K$
is an abelian extension of $K$ and contained in the ray class field
of conductor at most $\frakp^{i} \frakq$
because it is at most tamely ramified at $\frakq$.
But by assumption this ray class field must be trivial.
We do the same for the $q$-part of $L/K$.
Then it follows that $L/K$ is tamely ramified and that
$L$ is equal to the ray class field of conductor $\frakp \frakq$.
But this ray class field is contained in the ray class
field of conductor $\frakp^{i} \frakq$, which is trivial: contradiction.
\end{proof}

\begin{proposition} \label{prop:max3ext_grh}
Under assumption of the generalized Riemann hypothesis,
the maximal $3$-torsion extension of $\catd$ is given by $\maxextcatd = \Q(\calE[3],\sqrt[3]{2})$.
\end{proposition} 
\begin{proof}
Let $\maxextcatd$ be the maximal $3$-torsion extension of $\catd$.
The root discriminant of $\maxextcatd$ is bounded by $3^{3/2} 2^{5/2}=6^{3/2}\cdot 2 \approx 29.39$.
Under the assumption of the generalized Riemann hypothesis,
we apply Odlyzko's discriminant bounds in \cite[Table 3, p. 179]{Martinet:1982}
to find that $[\maxextcatd:\Q] < 1200$.
Hence $\maxextcatd$ is finite.
Group schemes annihilated by $3$ that we know of are $\calE[3]$, $\mu_{3}$, $\Z/3\Z$ and the
group schemes $T[3,2]$ and $T[3,-1]$ as in Definition \ref{def:katzmazur} .
We have the following extensions of number fields:
$$
\Q \subset_{2} \Q(i) \subset_{8} \Q(\calE[3]) \subset_{3} K := \Q(\calE[3],T[3,2]) = \Q(\calE[3],\sqrt[3]{2}) \subset_{\leq 25} \maxextcatd .
$$
The extension $\Q(\calE[3])/\Q$ is ramified of level $\frac{3}{2}$ at $2$ and tamely ramified at $3$.
The absolute discriminant of $\Q(\calE[3])$ is $2^{32} 3^{14}$.
The number field $K$ is equal to the ray class field of $\Q(\calE[3])$ of conductor $\bad_{2}\bad_{3}^{5}$
where $\bad_{2}$ and $\bad_{3}$ are the unique primes lying over resp. $2$ and $3$.
This tells us that $K/\Q$ is ramified of level $\frac{3}{2}$ at $2$, that $K/\Q$ is ramified of level
$\frac{1}{2}$ at $3$ and that the discriminant of $K$ is $3^{62} 2^{100}$.

\vskip 10pt
\noindent {\bf Claim 1}:
There does not exist a non-trivial Galois extension of $K$ that is
at most tamely ramified at the primes $\badover_{2}$ and $\badover_{3}$ in $K$ above resp. $2$ and $3$.
\begin{proof}
For any Galois extension $K'/K$ that is tamely ramified at both $\badover_{2}$ and $\badover_{3}$ we have:
$$
\Delta_{K'/K} = \badover_{2}^{f_{2}g_{2} (e_{2}-1)} \badover_{3}^{f_{3} g_{3} (e_{3}-1)} ,
$$
where $e_{i},f_{i}$ and $g_{i}$ are resp. the ramification, inertia and splitting degree of $K'/K$ at $\badover_{i}$ 
for $i=2,3$. Then
$$
N_{K/\Q}( \Delta_{K'/K} )^{\frac{1}{[K':\Q]}} = 2^{\frac{2}{[K:\Q]} (1-\frac{1}{e_{2}}) } 3^{\frac{2}{[K:\Q]} (1-\frac{1}{e_{3}})} 
< 6^{\frac{2}{[K:\Q]}} = 6^{\frac{2}{48}} .
$$
The numerator $2$ in the exponents in the second term
accounts for the fact that both the primes $2$ and $3$ have
inertial degree $2$ in $K/\Q$ and ramification degree $24$. 
Multiplying the root discriminant of $K$ with this value gives
$\delta_{K'} < 2^{\frac{102}{48}} 3^{\frac{64}{48}} \approx 18.873$.
Under the assumption of the generalized Riemann hypothesis, it follows from 
Odlyzko's discriminant bounds in \cite[Table 3, p. 179]{Martinet:1982} that $[K':\Q] < 96$.
This implies that $[K':K] < 2$ and therefore $K'=K$.
In particular, we see that the class number of $K$ is $1$.
\end{proof}

\noindent
{\bf Claim 2}: The group $\Gal(K/\Q(\calE[3]))$ is the maximal abelian quotient of $\Gal(\maxextcatd/\Q(\calE[3]))$.
\begin{proof}
In other words, the number field $K$ is the maximal abelian extension of $\Q(\calE[3])$ inside $\maxextcatd$.
Certainly $K/\Q(\calE[3])$ is abelian and it suffices to show that there is no larger abelian extension in $\maxextcatd$.
Therefore we consider ray class fields of conductor
of the form $\bad_{2}^{i} \bad_{3}^{j}$ with $i,j \in \N$, and show that they are either trivial or cannot exist inside $\maxextcatd$.
By Lemma \ref{lem:combinedconductors}, it is enough to consider conductors
for which either $i \in \{0,1\}$ or $j \in \{0,1\}$.

First we consider ray class fields of conductor $\bad_{2}^{i} \bad_{3}^{j}$ with $i \in \N$ and $j \in \{0,1\}$.
The smallest $i$ and $j$ that give a non-trivial ray class field turn out to be $i=8$ and $j=0$.
See the Appendix for this computation and the computations that will follow.
The ray class field of conductor $\bad_{2}^{8}$ is a degree $4$ extension of $\Q(\calE[3])$.
Consider the subextension $F/\Q(\calE[3])$ of degree $2$ inside this ray class field.
It admits a non-trivial character of conductor $\bad_{2}^{8}$, and using the conductor-discriminant
formula we see that the $2$-adic part of the root discriminant $\delta_{F}$ is equal to $\frac{5}{2}$,
contradicting the root discriminant bounds of Fontaine \cite[Corollary 3.3.2]{Fontaine:1985}.

Next, consider ray class fields of conductor $\bad_{2}^{i} \bad_{3}^{j}$ with $j \in \N$ and $i \in \{0,1\}$.
The smallest $i$ and $j$ that give a non-trivial ray class field turn out to be $i=1$ and $j=5$.
The ray class field of conductor $\bad_{2}^{1} \bad_{3}^{5}$ is exactly the field $K$. 
The next value of $j$ that gives a non-trivial extension is $j=8$.
Consider any subextension $F/\Q(\calE[3])$ of degree $3$ inside a ray class field of conductor $c$ such
that the $\bad_{3}$-part of $c$ has exponent at least $8$.
Then the $3$-adic part of the root discriminant $\delta_{F}$ is at least $\frac{16}{24}+\frac{14}{16}=\frac{37}{24}>\frac{3}{2}$ 
(where $\frac{16}{24}$ comes from the ramification in $F/\Q(\calE[3])$ and $\frac{14}{16}$ from
the ramification in $\Q(\calE[3])/\Q$).
This contradicts the fact that the $3$-adic valuation of the root discriminant should be smaller than $\frac{3}{2}$
by the bounds proven by Fontaine \cite[Corollary 3.3.2]{Fontaine:1985}.
Alternatively, we calculate the ramification at $\bad_3$ in the extension $F/\Q(\calE[3])$ in the following way:
The extension $\Q(\calE[3])/\Q$ is tamely ramified at $3$ and since the conductor at $\bad_{3}$ 
of $F/\Q(\calE[3])$ is $\bad_{3}^{8}$, we have that
$$
I_{F/\Q(\calE[3])}(3)_{i}  = 
\left\{
\begin{array}{ll}
3 & \text{ if } 0 \leq i \leq 8 \\
1 & \text{ if }  i > 8 .
\end{array}
\right.
$$
This implies that the extension $F/\Q$ is ramified of level $\frac{21}{24}$ at $3$, contradicting the fact that
the extension $\maxextcatd/\Q$ is only ramified of level $\frac{1}{2}$ at $3$
as follows from \cite[Theorem 2.1]{Fontaine:1985}.
\end{proof}

\noindent
{\bf Claim 3}:  The order of the maximal abelian quotient of $\Gal(\maxextcatd/K)$ is not divisible by $3$.
\begin{proof}
Consider the commutator subgroup $N$ of $\Gal(\maxextcatd/K)$
with fixed field $\maxextcatd'$.
The subgroup $N$ is normal in the group $\Gal(\maxextcatd/\Q(\calE[3]))$.
Therefore $\maxextcatd'/\Q(\calE[3])$ is Galois.
The $2$-sylow subgroup $H$ of
$\Gal(\maxextcatd'/K)$ is normal because $\Gal(\maxextcatd'/K)$ is abelian.
It follows that $H$ is normal in $\Gal(\maxextcatd'/\Q(\calE[3]))$.
Therefore the fixed field of $H$
corresponds to a Galois extension of $\Q(\calE[3])$ 
with a Galois group $P$.
The group $P$ is a $3$-group
and $P/P' = \Gal(K/\Q(\calE[3]))$ is cyclic of order $3$.
Hence $P = \Gal(K/\Q(\calE[3]))$.
\end{proof}

By Claim 1, the degree of 
the abelian quotient of $\maxextcatd/K$ must be a power of $2$ less than $25$, so either $2,4,8$ or $16$.

\vskip 10pt
\noindent
{\bf Claim 4}: Let $M=\Q(\zeta_{12},\sqrt[3]{2})$. Then $\Gal(K/M)$ is the maximal abelian quotient of $\Gal(\maxextcatd/M)$.
\begin{proof}
Let $\badover_{2}$ and $\badover_{3}$
denote the unique primes above respectively
$2$ and $3$ in $O_{M}$.
By Claim 3, the degree of any abelian subextension of $\maxextcatd/K$ cannot be divisible by $3$.
It follows that the maximal abelian quotient of $\Gal(\maxextcatd/M)$ is at most
tamely ramified at $3$ and therefore we only consider ray class fields
of conductor $\badover_{2}^i \badover_{3}$ for some positive integer $i$.

The extension $K/M$ is cyclic of order $4$.
There are two ray class fields over $M$ of degree $4$,
one of conductor $\badover_{2}^{7} \badover_{3}$ 
and one of conductor $\badover_{2}^{8}$.
They are disjoint over $M$ because one ray class field is totally ramified at $\badover_{3}$
while the other ray class field is unramified at $\badover_{3}$.
Since the ramification degree of $K/\Q$ at $3$ is $24$,  
it follows that $K/M$ is ramified at $\badover_{3}$.
Hence $K$ can be interpreted as the ray class field of conductor
$\badover_{2}^{7} \badover_{3}$.

The ray class field $R$ of conductor $\badover_{2}^{8} \badover_{3}$ is 
an extension of degree $16$ over $M$.
The Galois group $\Gal(R/M)$ is isomorphic to $\Z/4\Z \times \Z/4\Z$.
It is the compositum of the two ray class fields discussed above.
The extension $R/M$ admits three characters of conductor $\badover_{2}^{7} \badover_{3}$, three characters of conductor $\badover_{2}^{8}$,
nine characters of conductor $\badover_{2}^{8} \badover_{3}$ and one trivial character.
The conductor-discriminant formula then gives that $\Delta_{R/M} = \badover_{2}^{117} \badover_{3}^{12}$.
The root discriminant of the extension $M/\Q$ is equal to $\delta_{M/\Q} = 2^{\frac{16}{12}} 3^{\frac{14}{12}}$, 
and we calculate the root discriminant of the extension $R/\Q$ to be:
$$
\delta_{R/\Q} = 2^{\frac{117}{96}+ \frac{16}{12} } 3^{ \frac{14}{12} + \frac{12}{96} } = 2^{\frac{245}{96}}
3^{\frac{124}{96}} .
$$
But $\frac{245}{96} > \frac{5}{2}$, contradicting the bound on
the $2$-adic part of the root discriminant bound of the extension $\maxextcatd/\Q$.

This implies that any ray class field of conductor
$\badover_{2}^{i} \badover_{3}$ with $i \geq 8$ cannot be inside $\maxextcatd$.
It also implies that the ray class field of conductor
$\badover_2^8$ cannot occur inside $\maxextcatd$: otherwise
$\badover_2^8 \badover_3$ would be inside $\maxextcatd$.
This concludes the proof of the claim.
\end{proof}

We deduce, using exactly the same argument as given to prove Claim 3,
that the order of the abelian quotient of $\Gal(\maxextcatd/K)$ is not divisible by $2$.
We can now finish the proof of Proposition \ref{prop:max3ext_grh}. 
Since $2$ and $3$ are the only possible primes dividing the 
degree of the abelian quotient of $\Gal(\maxextcatd/K)$,
we conclude that there is no abelian extension inside $\maxextcatd/K$.
But $\Gal(\maxextcatd/K)$ is solvable, so it must be trivial as we wanted to show.
\end{proof}

To find the simple group schemes in $\catd$, we consider irreducible $\F_{3}[\shg]$-modules.
The Galois group of $\Q(\calE[3])/\Q$ is isomorphic to the semihedral group $\shg$ of order 16.
The group $\shg$ has presentation
$$
\langle s,t : s^{8} = t^{2} = 1 ,\, st = ts^{3} \rangle .
$$ 
The field $\F_{9}$ obtained by adjoining the $8$th roots of unity to $\F_{3}$ is a sufficiently large field
with respect to the group $\shg$ (see \cite[Chapter 14, p. 115]{Serre:1977}).
Since $\gcd(2,3)=1$, all $\F_{9}[\shg]$-modules are semi-simple.
There are $7$ conjugacy classes of the group $\shg$ (or $3$-regular conjugacy classes, cf.
\cite[Corollary, p. 150]{Serre:1977}). 
Therefore, we have that $\sum_{i=1}^{7} n_{i}^{2} = 16$, where $n_{i}$ is the degree of the $i$-th irreducible
submodule of $\F_{9}[\shg]$.
This is only possible if three simple modules have degree $2$
and the other four have degree $1$. 
The character table of these representations is as follows:

\begin{table}[H]
  \centering
  \begin{tabular}{@{} c|ccccccc @{}}
    length & 1 & 1 & 4 & 2 & 4 & 2 & 2  \\  
    ord & 1 & 2 & 2 & 4 & 4 & 8 & 8  \\ 
    
    \hline
    1 & 1 & 1 & 1 & 1 & 1 & 1 & 1 \\ 
    $\chi_{2}$ & 1 & 1 & -1 & 1 & -1 & 1 & 1 \\ 
    $\chi_{3}$ & 1 & 1 & -1 & 1 & 1 & -1 & -1 \\ 
    $\chi_{2} \chi_{3}$ & 1 & 1 & 1 & 1 & -1 & -1 & -1 \\ 
    $\rho_{\calE[3]}$ & 2 & -2 & 0 & 0 & 0 & -1 & 1 \\
    $\chi_{3 } \rho_{\calE[3]}$ & 2 & -2 & 0 & 0 & 0 & 1 & -1 \\ 
    $\rho_{g}$  & 2 & 2 & 0 & -2 & 0 & 0 & 0 \\ 
  \end{tabular}
  \caption{The character table of $\F_{3}[\shg]$.}
  \label{tab:q16char}
\end{table}

Length stands for the length of the conjugacy class and ord for the order of any element in a conjugacy class.
We check which of the above irreducible representations are flat over $\Z_3$:

\begin{definition}
A continuous representation $\rho: G_{\Q_p} \rightarrow \Gl_n(\fpbar)$ is \emph{flat over $\Z_p$}
if there exists a finite flat commutative group scheme $\gsc$ over $\Z_p$ such that
the $G_{\Q_p}$-action on $\gsc(\Alg{\Q}_p)$ is isomorphic to the representation $\rho$.
\end{definition}

\subsubsection*{Flat $1$-dimensional representations}

All the $1$-dimensional representations in Table \ref{tab:q16char} are flat:
The trivial character corresponds to the group scheme $\Z/3\Z$.
The cyclotomic character $\chi_3$ at the prime $3$ corresponds
to $\mu_3$.
The cyclotomic character $\chi_2$ at the prime $2$ is the generic fiber of the group scheme $\Z/3\Z(\chi_2)$.
The character $\chi_2 \chi_3$ is the generic fiber of the group scheme $\mu_3(\chi_2)$.

\subsubsection*{Flat $2$-dimensional representations}

The action of $\shg$ on $\calE[3]$ gives an irreducible
representation $\rho_{\calE[3]}$ of degree $2$.
We will prove that the only representation of the group $\shg$ of degree $2$
extending to a finite flat group scheme over $\Z_3$ is
$\rho_{\calE[3]}$.
The representation $\rho_{\calE[3]}$ has 
trace $1$ or $-1$ on elements of order $8$, because $\shg$
acts faithfully on $\calE[3]$ and all elements of
order $8$ in $\Gl_{2}(\F_{3})$ have trace $\pm 1$.
To prove that the remaining $2$-dimensional representations in Table \ref{tab:q16char} are not flat,
we use the theory of finite Honda systems (see \cite{Fontaine:1977}).

\begin{lemma}
The representations $\rho_{g}$ and $\chi_3 \rho_{\calE[3]}$ 
are not flat over $\Z_{3}$.
\end{lemma}
\begin{proof}
Since the $2$-dimensional representations in Table \ref{tab:q16char} are also irreducible
considered as representations of $\Gal(\Q_3(\calE[3])/\Q_3)$ (since $3$ does not split in this extension),
the finite flat group schemes corresponding to flat $2$-dimensional representations are simple.
Note that a simple finite flat commutative group scheme of $p$-power order
over $\Z_p$ is either local-local, local-\'etale or \'etale-local.
For local-local schemes we can apply \cite[Proposition 6.1]{Schoof:2003} with $p=3$ and $K=\Q_3$,
that tells us that there are exactly two non-isomorphic finite flat commutative
local-local group schemes over $\Z_3$ of order $9$.
The $\Alg{\Q}_3$-valued points of these two group schemes generate the field extensions
$\Q_3(i,\sqrt[8]{3})$ and $\Q_3(i,\sqrt[8]{6})$.
These are different extensions.

Assume that the representation $\chi_3 \rho_{\calE[3]}$
extends over $\Z_3$ to a finite flat commutative group scheme $\gsc'$.
The group scheme $\gsc'$ is a local-local group scheme
since the group scheme $\calE[3]$ over $\Z_3$ is a local-local group scheme.
The group schemes $\gsc'$ and $\calE[3]_{\Z_3}$ are not isomorphic
because their generic fibers are not isomorphic.
But we know that the representations $\chi_3 \rho_{\calE[3]}$ and
$\rho_{\calE[3]}$ factor through the same Galois group $\shg$,
which is contrary to the above. Hence the representation $\chi_3 \rho_{\calE[3]}$ is not flat
over $\Z_3$.

Similarly, assume that $\rho_{g}$ is a flat representation and extends to a 
finite flat commutative group scheme $\gsc''$ over $\Z_3$.
The representation $\rho_{g} : G_{\Q} \rightarrow \Gl_2(\F_{3})$ is not faithful
because the image has no elements of order $8$: 
All elements of order $8$ inside the unique subgroup
of $\Gl_{2}(\F_{3})$ of order $16$ have trace $\pm 1$,
yet $\rho_g$ has trace zero on elements of order $8$.
Hence the image is isomorphic to a quotient of $\shg$.
This quotient turns out to be isomorphic to the dihedral group $D_{4}$.
Therefore the extension generated by the points of $\gsc''$
cannot be $\Q_3(i,\sqrt[8]{3})$ or $\Q_3(i,\sqrt[8]{6})$,
since they both have degree $16$.
Hence $\gsc''$ is not local-local.
The Cartier dual of $\gsc''$ is also a simple finite flat
commutative group scheme of order $9$ which has an irreducible generic fiber
that is not isomorphic to $\rho_g$:
By \cite[Theorem 3.3.3]{Raynaud:1974}, the prolongation of $\rho_g$ must be unique since $1 < p-1$ holds. 
But the generic fiber of the Cartier dual of $\gsc''$ does not occur in Table \ref{tab:q16char}.
Therefore we conclude that $\rho_g$ is not flat.
\end{proof}

\begin{proposition} \label{prop:j032_simple_groupschemes}
Any simple group scheme in $\catd$ is isomorphic to either 
$\Z/3\Z$, $\Z/3\Z(\chi_{2})$, $\mu_{3}$, $\mu_{3}(\chi_{2})$ or $\calE[3]$.
\end{proposition}
\begin{proof}
It is clear that the group schemes in the statement are objects $\catd$.
We prove that they are the only ones up to isomorphism.
A simple group scheme in $\catd$ prolongates the
$1$-dimensional representations in Table \ref{tab:q16char} or the representation
$\rho_{\calE[3]}$.
In particular, all simple group schemes have
either order $3$ or $9$.
We start with the group schemes of order $9$.

The generic fiber of a simple group scheme of order $9$ in $\catd$
is isomorphic to the generic fiber of $\calE[3]$.
Suppose that $G$ is a simple object in $\catd$ and has order $9$.
The local group scheme $G \times_{\Z[\frac{1}{2}]} \Z_{3}$ prolongs $\calE[3] \times_{\Z[\frac{1}{2}]} \Q_{3}$.
According to \cite[Proposition 3.3.2]{Raynaud:1974}, 
the group scheme $G \times_{\Z[\frac{1}{2}]} \Z_{3}$ is isomorphic to $\calE[3] \times_{\Z[\frac{1}{2}]} \Z_{3}$.

We apply the result of \cite{Artin:1970} described in \cite[Proposition 2.3]{Schoof:2003}.
The scheme $\gsc$ corresponds to a triple $(\gsc',\gsc'',\theta_{\gsc})$
where $\gsc' = \gsc \times_{\Z[\frac{1}{2}]} \Q$,
$\gsc'' = \gsc \times_{\Z[\frac{1}{2}]} \Z_{3}$
and where $\theta_{\gsc} : \gsc' \times_{\Q} \Q_{3} \rightarrow \gsc'' \times_{\Z_{3}} \Q_{3}$ is a gluing homomorphism.
Similarly, write $\calE[3]$ as $(\calE[3]_{1},\calE[3]_{2},\theta_{\calE[3]})$.
By hypothesis, there exist isomorphisms 
$f_1 : \gsc' \rightarrow \calE[3]_{1}$ and 
$f_2 : \gsc'' \rightarrow \calE[3]_{2}$.
We want that $f_1$ and $f_2$, when base changed to $\Q_{3}$,
are compatible with $\theta_{\calE[3]}$ and $\theta_\gsc$.
Locally over $\Q_{3}$ we can always do this:
we can adjust $(f_2)_{\Q_{3}}$ by composing it with
an automorphism of $\calE[3]_2 \times_{\Z[\frac{1}{2}]}  \Q_3$.
The only $G_\Q$-equivariant (resp. $G_{\Q_3}$-equivariant)
automorphism of $\calE[3]_2$ (resp. $\calE[3]_2 \times_{\Z_{3}} \Q_{3}$) is the identity.
This implies that $\gsc$ is isomorphic to $\calE[3]$. 
It follows that there is only one simple group scheme
of order $9$ up to isomorphism in $\catd$, namely 
the group scheme $\calE[3]$.

We now consider group schemes of order $3$.
The generic fiber of a group scheme $\gsc$ of order $3$
is isomorphic to the generic fiber of one of
the group schemes
$\Z/3\Z, \Z/3\Z(\chi_{2}), \mu_{3}$ or $\mu_{3}(\chi_{2})$.
We apply \cite[Theorem 3]{TateOort:1970} that says that
the generic fiber determines $\gsc$ up to a twist.
Such a twist is only ramified at $2$ and has order divisible
by the order of $\F_3^*$, the automorphism group of $\gsc \times_{\Z[\frac{1}{2}]} \Z_{3}$.
Hence the twist must be a power of $\chi_2$.
\end{proof}

\subsection{Extensions of simple group schemes} \label{sec:extj032}

Now that we classified the simple group schemes
in $\catd$, we compute extension groups
of simple group schemes in $\catd$.
In particular, we will show that any extension in $\catd$
of a simple non-\'etale group scheme by a simple \'etale group scheme, that is, $\Z/3\Z$
or $\Z/3\Z(\chi)$, is trivial. 
To show this for the category $\catd$, 
we will use that, under assumption of the generalized Riemann hypothesis,
the maximal $3$-torsion extension $\maxextcatd$ of $\catd$ is equal to $\Q(\calE[3],\sqrt[3]{2})$.
The results that we obtain using this, will therefore a fortiori 
also hold under assumption of the generalized Riemann hypothesis.
When calculating extension groups, we make use of the following result:

\begin{proposition}[\cite{Schoof:2003}, Corollary 2.4] \label{prop:mayervietoris}
Let $\badrat, \good$ be two distinct rational primes.
Let $\gsc'$ and $\gsc''$ be two finite flat commutative group schemes
over $\Z[\frac{1}{\badrat}]$. 
The sequence
\begin{align*}
0 \longrightarrow \Hom_{\Z[\frac{1}{\badrat}]}(\gsc'',\gsc') \longrightarrow 
\Hom_{ \Z_\good  }(\gsc'',\gsc') \times
\Hom_{  \Z[\frac{1}{\good \badrat}] }(\gsc'',\gsc') \longrightarrow
\Hom_{  \Q_\good }(\gsc'',\gsc') \\ 
\longrightarrow \extt{ \Z[\frac{1}{\badrat}]  }{\gsc''}{\gsc'}
\longrightarrow  \extt{ \Z_\good}{\gsc''}{\gsc'} \times
\extt{ \Z[\frac{1}{\good \badrat}] }{\gsc''}{\gsc'} \longrightarrow 
\extt{ \Q_\good }{\gsc''}{\gsc'}
\end{align*}
is exact.
\end{proposition}

This allows us to relate extensions over $\Z[\frac{1}{2}]$ with
local extensions and extensions of Galois modules,
as will be demonstrated in the next lemma.

\begin{lemma} \label{lem:ext_e3_z3z_grh}
Under assumption of the generalized Riemann hypothesis,
the $\F_3$-vector spaces \\
$\extt{\Z[\frac{1}{2}]}{\calE[3]}{\Z/3\Z}$ and 
$\extt{\Z[\frac{1}{2}]}{\calE[3]}{\Z/3\Z(\chi_2)}$ are trivial.
\end{lemma}
\begin{proof}
Let $T$ be one of the group schemes $\Z/3\Z$ or $\Z/3\Z(\chi_2)$.
Since $\Gal(\Q(\calE[3])/\Q)$ is equal to the decomposition group at $3$,
it follows that $\Hom_{\Q}(\calE[3], T) \simeq \Hom_{\Q_3}(\calE[3], T)$.
Because $\Hom_{\Q}(\calE[3], T)$ is trivial,  
Proposition \ref{prop:mayervietoris} with $\gsc'' = \calE[3]$ and $\gsc' = T$ 
implies that  $\extt{\Z[\frac{1}{2}]}{\calE[3]}{T}$ is mapped injectively
into $\extt{\Z[\frac{1}{6}]}^{1}{\calE[3]}{T} \times \extt{\Z_3}{\calE[3]}{T}$.
Let $\gsc$ be an extension of $\calE[3]$ by $T$
of group schemes over $\Z[\frac{1}{2}]$.
Then $\gsc$ splits locally at $3$ because $\calE[3]$ is connected and $T$ is \'etale.
Thus $\gsc_{\Z_{3}}$ is annihilated by $3$.
Because $\Z_3$ is flat over $\Z[\frac{1}{2}]$, the group scheme
$\gsc$ is also annihilated by $3$.
The group scheme $\gsc$ is also generically split by the following argument.
The number field $\Q(J)$ contains $\Q(\calE[3])$.
Since $\gsc$ is annihilated by $3$, the extension
$\Q(J)/\Q$ must be a subextension of the maximal extension $\Q(\calE[3],\sqrt[3]{2})$.
Hence $\Q(\gsc)$ is either $\Q(\calE[3],\sqrt[3]{2})$ or $\Q(\calE[3]))$.
Since $\gsc$ is locally a trivial extension, the Galois extension $\Q(J)/\Q(\calE[3])$
is unramified at the prime lying above $3$ in $\Q(\calE[3])$.
However,
the extension $\Q(\calE[3],\sqrt[3]{2})/\Q(\calE[3])$ is ramified 
at the prime lying above $3$.
Therefore $\Q(J) = \Q(\calE[3])$.
It follows that $J_\Q$ is an $\F_3[\shg]$-module.
Since $\F_3[\shg]$-modules are semi-simple, the sequence splits generically.
Hence also $\gsc$ is a split extension.
\end{proof}

\begin{lemma} \label{lem:ext_twisted}
The $\F_3$-vector spaces
\begin{eqnarray*}
& \extt{\Z[\frac{1}{2}]}{\mu_3}{\Z/3\Z} ,\quad \extt{\Z[\frac{1}{2}]}{\mu_3}{\Z/3\Z(\chi_2)} , \quad \extt{\Z[\frac{1}{2}]}{\mu_3(\chi_2)}{\Z/3\Z)} & \\
& \text{and} \quad \extt{\Z[\frac{1}{2}]}{\mu_3(\chi_2)}{\Z/3\Z(\chi_2)}  &
\end{eqnarray*}
are trivial.
\end{lemma}
\begin{proof}
It follows from \cite{Verhoek:2010} that the $\F_3$-vector space
$\extt{\Z[\frac{1}{2},i]}{\mu_3}{\Z/3\Z}$ is trivial.
Let $\gsc''$ be either $\mu_3$ or $\mu_3(\chi_2)$
and let $\gsc'$ be either $\Z/3\Z$ or $\Z/3\Z(\chi_2)$.
Since every extension of $\gsc''$ by $\gsc'$ splits locally,
the long exact sequence of Proposition \ref{prop:mayervietoris}
yields the inclusion
$
\extt{\Z[\frac{1}{2}]}{\gsc''}{\gsc'} \hookrightarrow 
\extt{\Q}{\gsc''}{\gsc'}
$.
It suffices to show that $\extt{\Q}{\gsc''}{\gsc'}$ is trivial.
Let $\Gamma = \Gal(\Q(i)/\Q)$.
We have the following exact sequence:
$$
0 \longrightarrow H^1(\Gamma,\Hom_{G_\Q}(\gsc'',\gsc'))
\longrightarrow \extt{\Q}{\gsc''}{\gsc'} \longrightarrow 
\extt{\Q(i)}{\gsc''}{\gsc'}^{\Gamma} .
$$
By the above,  $\extt{\Q(i)}{\gsc''}{\gsc'}^{\Gamma}$ is trivial.
The cohomology group $H^1(\Gamma,\Hom_{G_\Q}(\gsc'',\gsc'))$
is trivial because $\Hom_{G_\Q}(\gsc'',\gsc')$ is trivial.
\end{proof}

\begin{lemma} \label{lem:ext_3tors_3tors_grh}
Under the assumption of the generalized Riemann hypothesis,
the group of extensions $\extt{\Z[\frac{1}{2}]}{\calE[3]}{\calE[3]}$
has order $3$ and is generated by $\calE[9]$.
\end{lemma}
\begin{proof}
By Lemma \ref{lem:end_gal_equivariant} the group $\Hom_{\Q}(\calE[3],\calE[3])$ 
has order $3$.
The same is true for $\Q_3$ instead of $\Q$ because $3$ does not split
in $\Q(\calE[3])$.
In the long exact sequence of Proposition \ref{prop:mayervietoris}, the map
$$
\Hom_{\Q}(\calE[3],\calE[3]) \times
\Hom_{ \Z_3}(\calE[3],\calE[3]) \longrightarrow
\Hom_{\Q_3}(\calE[3],\calE[3])
$$
is surjective.
As a consequence we get the following inclusion: 
\begin{align} \label{eqn:extincl2}
\extt{\Z[\frac{1}{2}]}^{1}{\calE[3]}{\calE[3]} \longhookrightarrow
\extt{\Z_3}^{1}{\calE[3]}{\calE[3]} \times \extt{\Z[\frac{1}{6}]}^{1}{\calE[3]}{\calE[3]} .
\end{align}
Let $\gsc$ be an extension of $\calE[3]$ by $\calE[3]$ 
of group schemes over $\Z[\frac{1}{2}]$
and suppose that $\gsc$ is annihilated by $3$.
By Proposition \ref{prop:max3ext_grh}
we know that $\Q(\gsc) \subset \Q(\calE[3],\sqrt[3]{2})$.
Apply \cite[Proposition 6.4]{Schoof:2003} with $p=3$,
that says that the number field $\Q(\gsc)$ is either:
$\Q(\calE[3])$, an unramified extension of $\Q(\calE[3])$ or an extension of degree 
at least $9$ over $\Q(\calE[3])$.
We conclude that the only possibility is that 
$\Q(\gsc) = \Q(\calE[3])$.
Since all $\F_3[\shg]$-modules are semi-simple, the group scheme
$\gsc$ splits generically.

Next, we consider the group scheme $\gsc_{\Z_3}$.
It is an extension of $\calE[3]_{\Z_3}$
by itself.
According to \cite[Theorem 3.3.3]{Raynaud:1974}, the local group scheme
$\gsc_{\Z_3}$ is uniquely determined
by its generic fiber because the condition $e < p-1$ holds,
where $e$ is the ramification at $3$ of $\Q_3/\Q_3$ and where $p=3$.
But $\gsc_{\Z_3}$ splits generically, that is, $\gsc_{\Q_3}$ is the product of $\calE[3]_{\Q_3}$ with itself because
the Galois action is trivial and $\gsc_{\Q_3}$ is annihilated by $3$.
Hence by \cite[Theorem 3.3.3]{Raynaud:1974} the group scheme $\gsc_{\Z_3}$ splits.
Now use inclusion (\ref{eqn:extincl2}) above to deduce that extensions annihilated by $3$ split, i.e.,
the group $\extt{\Z[\frac{1}{2}],[3]}^{1}{\calE[3]}{\calE[3]}$ is trivial.

By \cite[Lemma 2.1]{Schoof:2009}, the following sequence is exact:
$$
0 \longrightarrow \extt{\Z[\frac{1}{2}],[3]}{\calE[3]}{\calE[3]}  \longrightarrow \extt{\Z[\frac{1}{2}]}{\calE[3]}{\calE[3]} \longrightarrow 
\Hom_{G_{\Q}}(\calE[3],\calE[3]) .
$$
Since $\Hom_{G_{\Q}}(\calE[3],\calE[3])$ has order $3$ and $\extt{\Z[\frac{1}{2}],[3]}^{1}{\calE[3]}{\calE[3]}$ is trivial,
the group scheme $\calE[9]$ generates the group $\extt{\Z[\frac{1}{2}]}{\calE[3]}{\calE[3]}$.
\end{proof}

\subsection{Proof of Theorem \ref{thm:mainj032}} \label{sec:modularityj032}

\begin{proof}
Let $A$ be an abelian variety over $\Q$
with good reduction outside $2$
such that $A_\locfieldext$ has good reduction (resp. such that $A$ is ramified of level $\frac{3}{2}$ at $2$).
Let $\calA$ be the N\'eron model of $A$ over $\Z[\frac{1}{2}]$ and 
let $\calA[3]$ be the finite flat commutative
group scheme of $3$-torsion points of $\calA$.
Then $\calA[3]$ is an object in $\catd$ and has order $3^{2g}$.
Because any extension in $\catd$ of a simple
non-\'etale group scheme and an \'etale group scheme is split,
and because the ray class field of $\Q(i)$ of conductor $3$ is cyclic of degree $2$,
we can apply Theorem \ref{cat-thm:abvar_torsionfilter}
that says that $\calA[3]$ does not contain \'etale or multiplicative subquotients.
Hence, the only simple subquotients of $\calA[3]$ are isomorphic to $\calE[3]$.

To prove the theorem, we show that the $3$-divisible groups of $A$ and $E$ are isomorphic.
For this we apply \cite[Theorem 8.3]{Schoof:2005} and verify that
all conditions of Theorem \cite[Theorem 8.3]{Schoof:2005} are satisfied.
Let $\calA[3^{\infty}]$ be the $3$-divisible group associated to $A$ and let
$\calE[3^{\infty}]$ be the $3$-divisible group associated to $E$. 
By Lemma \ref{lem:end_gal_equivariant} and Tate's Theorem \cite[Corollary 1 of Theorem 4]{Tate:1967},
we have $\End(\calE[3^{\infty}]) \simeq \Z_{3}$.
Together with Lemma \ref{lem:ext_3tors_3tors_grh} we see that 
both $\calE[3^{\infty}]$ and $\calA[3^{\infty}]$ satisfy all hypotheses
of \cite[Theorem 8.3]{Schoof:2005}
and so $\calA[3^{\infty}] \simeq \calE[3^{\infty}]^{\dim(A)}$.
By \cite{Faltings:1983}, we deduce that $A$ is isogenous to $J_{0}(32)^{\dim(A)}$.
\end{proof}

\section{Good reduction outside $3$} \label{sec:elliptic_curve}

In this section, we prove the part of Theorem \ref{maintheorem} regarding good reduction outside $3$:

\begin{theorem} \label{j027-thm:main}
Let $A$ be an abelian variety over $\Q$ with good reduction outside $3$
such that $A$ is ramified of level $\frac{1}{2}$ at $3$.
Then $A$ is isogenous to a product of copies of $J_{0}(27)$.
\end{theorem}

\begin{definition}
Let $\catb$ be the category whose objects are finite flat commutative group schemes
of $2$-power order over $\Z[\frac{1}{3}]$ that are ramified
of level $\frac{1}{2}$ at the prime $3$ and whose morphisms
are group scheme morphisms.
\end{definition}

To prove Theorem \ref{j027-thm:main}, we consider
the abelian variety attached to the newform of level $27$ and weight $2$.
This abelian variety is an elliptic curve of conductor $27$.
Because Theorem \ref{j027-thm:main} only considers abelian varieties up to isogeny, 
we are only interested in abelian varieties up to isogeny.
We may therefore take any elliptic curve of conductor $27$.
Here we pick the elliptic curve labeled 27A3 in the tables of \cite{BirchKuyk:1975}:
$$
E : y^{2} + y = x^{3}  .
$$
Let $\calE$ be the N\'eron model of $E$ 
over $\Z[\frac{1}{3}]$.
The elliptic curve $E$ has additive reduction at the prime $3$.
Since the curve has complex multiplication over $\Q(\zeta_3)$,
it has potential good reduction. 
By the criterion of N{\'e}ron-Ogg-Shafarevich,
it obtains good reduction over the field $\Q(\calE[4])$ of degree $24$.
Furthermore, a short calculation shows that the Galois group through which $\rho_{\calE[2]}$
factors is isomorphic to the symmetric group $S_3$.

\begin{lemma} \label{lem:end_gal_equivariant2}
We have $\End_{G_{\Q}}(T_{2}(E)) \simeq \Z_{2}$.
In particular,
$\End_{G_{\Q}}(\calE[2])$ is cyclic of order $2$.
\end{lemma}
\begin{proof}
Since $E$ is supersingular at $2$ and has good reduction at $2$,
it follows from \cite[Theorem 1.1]{Conrad:1997} that $\overline{\rho_{E,2}}$ is absolutely irreducible.
We deduce that also $\rho_{E,2}$ is absolutely irreducible.
Now apply Schur's Lemma, see for instance \cite[Paragraph 4, Corollary, p. 252]{Mazur:1997}.
\end{proof}

We see from Table \ref{tab:modformlevel} that $J_{0}(27)$ is ramified of level $\frac{1}{2}$ at $3$
and hence that $\calE[2^{n}]$ is an object in the category $\catb$ for all positive integers $n$. 
Hence we will work in the category $\catb$.

\subsection{Simple group schemes} \label{j027-sec:simple}

We find all simple group schemes of $\catb$ using the maximal
$2$-torsion extension $\maxextcatb$ extension of the category $\catb$.
Recall that this extension is defined to be the compositum of all number fields
generated by the $\Alg{\Q}$-points of group schemes in $\catb$ that are killed by $2$.

\begin{lemma} \label{j027-lem:conductorsubfield}
Let $\good$ be a prime, $K$ be a number field and let $L/K$ be an extension contained in the ray class field
of conductor $c$ over $K$, such that $L$ is not contained in any ray class field of conductor strictly dividing $c$.
If $[L:K]=p$, then the discriminant $\Delta_{L/K}$ is equal to $c^{p-1}$.
\end{lemma}
\begin{proof}
Immediate by using the conductor-discriminant formula.
\end{proof}

\begin{lemma} \label{j027-lem:max2ext}
The extension $\Q(\calE[4])$ is the largest extension of $\Q$ that is 
ramified of level $\frac{1}{2}$ at $3$ and ramified of level $1$ at $2$.
The maximal $2$-torsion extension $\maxextcatb$ of $\catb$ is contained in $\Q(\calE[4])$. 
\end{lemma} 
\begin{proof}
The latter statement follows immediately from the first.
Let $L$ be the largest extension of $\Q$ that is 
ramified of level $\frac{1}{2}$ at $3$ and ramified of level $1$ at $2$.
We will prove that $L$ is equal to $\Q(\calE[4])$.
From the fact that group schemes annihilated by $2$ are ramified of level $\frac{1}{2}$ at $3$,
we have by Proposition \ref{prop:actionrootdiscb} that $\delta_{L} < 4 \cdot 3^{3/2} \approx 20.785$.
We apply Odlyzko's discriminant bounds (see \cite[Table 4, p. 187]{Martinet:1982}) to find that $[L:\Q] < 900$.
Hence $L/\Q$ is finite and in particular, the extension $\maxextcatb/\Q$ is finite.

Group schemes in $\catb$ annihilated by $2$ that we know of are $\calE[2]$, $\mu_{2}$, $\Z/2\Z$,
$T[2,-1]$ and $T[2,3]$.
Consider the following inclusions:
$$
\Q \subset_{2} \Q(\zeta_{3}) \subset_{3} \Q(\calE[2]) =
\Q(\zeta_{3},\sqrt[3]{4}) \subset_{2} \Q(\zeta_{3},\sqrt[3]{4}, i) \subset_{2} M := \Q(\calE[4]) \subset_{\leq 37} L .
$$
We explain why $M = \Q(\calE[4])$ is contained in $L$.
Consider the ramification at $2$ of $\calE[4]$.
By a computation using Magma (see the Appendix),
one finds that the orders of the ramification groups with lower numbering
are
$$
\# I_{\Q(\calE[4]) /\Q}(2)_{i} =
\left\{
\begin{array}{ll}
24 & \text{ if } i = 0 \\
12 & \text{ if }  i =1 \\
4 & \text{ if }  1 < i  \leq 4 \\
1 & \text{ if }  i > 4 .
\end{array}
\right. 
$$
We verify immediately that $\Q(\calE[4])$ is ramified of level $1$ at $2$.
Hence the group scheme $\calE[4]$ satisfies the same bounds on the 
ramification at $2$ as group schemes annihilated by $2$.
Since $[L:M]\leq 37$, the group $\Gal(L/M)$ is solvable and we can apply class field theory to find $L$.
Therefore, by Lemma \ref{lem:combinedconductors}, we consider
ray class fields of conductor $\badover_{2} \badover_{3}^{i}$ and $\badover_{2}^{j} \badover_{3}$,
where $\badover_{2}$ and $\badover_{3}$ denote the unique primes above $2$ resp. $3$ in $M$
and where $i, j$ are positive integers.
We show that for all $i,j \in \N$ there are no ray class fields of such conductors,
that are strictly larger than $M$ and that are contained in $L$.

We computed (see Appendix) that 
the ray class groups of conductor $\badover_{2}^j \badover_{3} $ with $j \leq 9$ are trivial
and that there is a degree $2$ extension ${L}_{3}/M$ of conductor $\badover_{2}^{10} \badover_{3}$.
Lemma \ref{j027-lem:conductorsubfield} applied with $c=\badover_{2}^{10} \badover_{3}$ and $p=2$
shows that the root discriminant $\delta_{{L}_{3}}$ has $2$-adic valuation $100/48 \geq 2$.
This contradicts the upper bounds given in \cite[Corollary 3.3.2]{Fontaine:1985}.
We conclude that all subfields of ray class fields of conductor $\badover_{2}^j \badover_{3}$,
for $j$ a positive integer, cannot be contained in $L$.

We also computed that the ray class fields of conductor $\badover_{2} \badover_{3}^{i}$ are trivial for $1 \leq i \leq 5$.
Suppose there is an abelian extension $L_3$ of $M$ in $L$ of degree $3$ and inside a ray class field
of conductor $\badover_{2} \badover_{3}^{6}$ (or even with a higher power than $6$).
Then, by Lemma \ref{j027-lem:conductorsubfield},
the $3$-adic valuation of the root discriminant of $L_3$ would be at least $\frac{5}{4} + \frac{1}{3} > \frac{3}{2}$.
This is again in contradiction to the upper bounds given in \cite[Corollary 3.3.2]{Fontaine:1985}.
\end{proof}

Before we continue to classify all simple group schemes in $\catb$, we prove
a lemma that tells us when an $\F_2[S_3]$-module is semi-simple.
Not all $\F_2[S_3]$-modules are semi-simple:
Consider for example the $\F_2[S_3]$-module $\F_2[\epsilon]/\epsilon^2$ with trivial action of $\sigma$
and $\tau$ acting as multiplication by $\epsilon$.

\begin{lemma} \label{j027-lem:f2s3_modules}
Let the symmetric group $S_3$ be represented by 
$$
\langle \sigma , \tau : \sigma^3=1 \, , \tau^2=1 \, , \tau \sigma \tau = \sigma^2 \rangle .
$$
There are exactly two irreducible $\F_2[S_3]$-modules:
the module $\F_2$ of order $2$ and the $\F_2$-algebra in $M_2(\F_2)$
of order $4$ generated by $\psi(\sigma)$, where $\psi$ is the isomorphism
$
\F_2[S_3]/(\sigma^2 + \sigma + 1)\F_2[S_3] \stackrel{\sim }{\longrightarrow}  M_2(\F_2)
$.
Moreover, any $\F_2[S_3]$-module that has at most
one irreducible non-zero submodule with trivial $\sigma$ action is semi-simple.
\end{lemma} 
\begin{proof}
It suffices to consider simple modules with either trivial action of $\sigma$
or annihilated by the $\sigma$-norm $1+\sigma + \sigma^2$ with non-trivial $\sigma$-action,
since for any $\F_2[S_3]$-module $T$ we have
$T \simeq T /(\sigma-1)T \times T/(\sigma^2 + \sigma + 1)T$.
The modules with trivial action of $\sigma$ are $\F_2[\tau]$-modules.
Such a simple module is isomorphic to $\F_2$.
Next we consider simple modules with non-trivial $\sigma$-action that are
annihilated by the $\sigma$-norm.
One checks that the $\F_2$-algebra in $M_2(\F_2)$
generated by $\psi(\sigma)$ is simple, 
has non-trivial action of $\sigma$ and has order $4$.
On the other hand, any other non-zero
submodule of $M_2(\F_2)$ that has non-trivial action of $\sigma$
is the non-simple $\F_2$-algebra $M_2(\F_2)$ itself.

For the second part of the statement, it suffices to prove that an $\F_2[S_3]$-module $T$
such that $\sigma$ acts non-trivially on all non-zero submodules of $T$, is semi-simple.
Such modules are $M_2(\F_2)$-modules.
The rings $M_2(\F_2)$ and $\F_2$ are Morita equivalent, see \cite[Corollary 22.6, p. 265]{AndersonFuller:1992}.
Morita equivalent rings $R$ and $S$ behave well with respect to semi-simplicity:
All $R$-modules are semi-simple if and only if all $S$-modules are semi-simple, see \cite[Proposition 21.8, p. 258]{AndersonFuller:1992}.
Since $\F_2$-modules are semi-simple, the result follows.
\end{proof}

\noindent
Our preliminary work above now allows us to classify all simple objects in the category $\catb$:

\begin{proposition}
The simple group schemes in $\catb$ are isomorphic to either $\mu_{2}, \Z/2\Z$ or $\calE[2]$.
\end{proposition}
\begin{proof}
Let $\gsc$ be a simple group scheme in $\catb$.
Its generic fiber $\gsc_{\Q}$ corresponds to an
irreducible Galois module.
By Lemma \ref{j027-lem:max2ext},
the group scheme $\gsc_{\Q}$ corresponds to an irreducible $\F_{2}[S_{3}]$-module.
By Lemma \ref{j027-lem:f2s3_modules}, there are only two irreducible $\F_{2}[S_{3}]$-modules, one of order $2$ and one of order $4$.
Therefore $\gsc$ has either order $2$ or order $4$.
For group schemes of order $2$ we use \cite{TateOort:1970}: 
the group scheme $\gsc$ is isomorphic to either
$\mu_{2}$ or $\Z/2\Z$.

It remains to classify the simple group schemes of order $4$.
Suppose that $\gsc$ has order $4$.
Then $\gsc_{\Z_{2}}$ as well as
the local group scheme $\calE[2]_{\Z_2}$ prolong
the generic fiber of $\calE[2]_{\Z_2}$.
The group scheme $\calE[2]_{\Z_2}$
is local-local because the elliptic curve $E$ is supersingular at $2$.
Therefore $\calE[2]_{\Z_2}$ is neither \'etale nor multiplicative.
According to \cite[Proposition 3.3.2]{Raynaud:1974}, 
the group scheme $\calE[2]_{\Z_2}$ has to be the unique group scheme up to isomorphism
that prolongs its own generic fiber.
Hence $\calE[2]_{\Z_2} \simeq \gsc_{\Z_2}$.

To prove that $\gsc \simeq \calE[2]$,
we apply the result of \cite{Artin:1970} as described in \cite[Proposition 2.3]{Schoof:2003}:
the group scheme $\gsc$ can be described by a triple $(\gsc_1,\gsc_2,\theta_{\gsc})$,
where $\gsc_1$ is a group scheme over $\Q$, $\gsc_2$ a group scheme over $\Z_2$ and
where $\theta_{\gsc}$ is a glueing morphism of group schemes over $\Q_2$. 
Similarly for $\calE[2]$ we have $(\calE[2],\calE[2],\theta_{\calE[2]})$.
Since the prime $2$ doesn't split in $\Q(\calE[2])/\Q$ and
the only $G_\Q$-equivariant (or $G_{\Q_2}$-equivariant)
automorphism of $\calE[2]$ is the identity,
we have $\gsc \simeq \calE[2]$.
It follows that $\calE[2]$ is the only simple group scheme
of order $4$ in the category $\catb$ up to isomorphism.
\end{proof}

\subsection{Extensions of simple group schemes} \label{j027-sec:extensions}

We consider extensions of the simple group schemes in the category
$\catb$ and we will prove 
that any extension in $\catb$
of a simple non-\'etale group scheme by $\Z/2\Z$ is split.

\begin{proposition}  \label{j027-prop:ext_mu2_z2}
The $\F_{2}$-vector spaces $\extt{\Z[\frac{1}{3}]}{\calE[2]}{\Z/2\Z}$ and  $\extt{\Z[\frac{1}{3}]}{\mu_{2}}{\Z/2\Z}$
are trivial.
\end{proposition}
\begin{proof}
That $\extt{\Z[\frac{1}{3}]}{\mu_{2}}{\Z/2\Z}$ is trivial follows from
\cite[Corollary 4.2]{Schoof:2005}.
To prove that the group $\extt{\Z[\frac{1}{3}]}{\calE[2]}{\Z/2\Z}$ is trivial, 
we use the long exact sequence of Proposition \ref{prop:mayervietoris}
and the fact that both $\Hom_{\Q}(\calE[2], \Z/2\Z)$ and $\Hom_{\Z_{2}}(\calE[2], \Z/2\Z)$ are trivial.
We deduce that the map
\begin{equation} \label{eqn:j027_ext_mu2_z2}
\extt{\Z[\frac{1}{3}]}{\calE[2]}{\Z/2\Z}
\longhookrightarrow
\extt{\Z_{2}}^{1}{\calE[2]}{\Z/2\Z} \times \extt{\Z[\frac{1}{6}]}{\calE[2]}{\Z/2\Z} .
\end{equation}
is an injection.
We show that the image
of $\extt{\Z[\frac{1}{3}]}{\calE[2]}{\Z/2\Z}$ in this product is trivial.
The sequence 
$$
0 \longrightarrow \Z/2\Z \longrightarrow \gsc \longrightarrow \calE[2] \longrightarrow 0 
$$
splits locally because the group scheme $\calE[2]$ is connected and $\Z/2\Z$ is \'etale. 
So $\gsc_{\Z_{2}}$ is annihilated by $2$. 
Because $\Z_{2}$ is flat over $\Z[\frac{1}{3}]$, also the group scheme $\gsc$ over 
$\Z[\frac{1}{3}]$ is annihilated by $2$.

The group scheme $\gsc$ is generically split:
the degree of $\Q(\gsc)/\Q(\calE[2])$ divides $2$, hence in particular $\Q(\gsc)/\Q(\calE[2])$ is abelian.
Since the extension $\gsc_{\Z_2}$ is trivial, $\Q(\gsc)/\Q(\calE[2])$ can only ramify at $3\infty$.
But the class number of $\Q(\calE[2])$ is $1$ and there are no abelian extensions over $\Q(\calE[2])$ with tame ramification at $3$.
Hence $\Q(\gsc) = \Q(\calE[2])$ and $\gsc_\Q$ is an $\F_2[S_3]$-module.
By Lemma \ref{j027-lem:f2s3_modules}, this module is semi-simple
and we deduce that the sequence splits generically.

Hence by the injectivity of (\ref{eqn:j027_ext_mu2_z2}), the group scheme $\gsc$ splits
and we proved that the $\F_2$-vector space $\extt{\Z[\frac{1}{3}]}{\calE[2]}{\Z/2\Z}$ is trivial.
\end{proof}

\begin{lemma} \label{j027-lem:ext_2tors_2tors}
The vector space $\extt{\Z[\frac{1}{3}]}{(\calE[2]}{\calE[2]}$ is generated by the non-trivial extension $\calE[4]$.
\end{lemma}
\begin{proof}
A non-trivial extension of $\calE[2]$ by itself
that is of order $4$ and that is not annihilated by $2$,
is given by the group scheme $\calE[4]$.
It remains to show that there is no other non-trivial extension.
By Lemma \ref{lem:end_gal_equivariant2},
the endomorphisms of the Galois modules $\calE[2](\Alg{\Q})$ and $\calE[2](\Alg{\Q}_{2})$
are scalar multiplications.
Thus the group $\Hom_{\Q}(\calE[2],\calE[2])$ is of order $2$,
and by \cite[Lemma 2.1]{Schoof:2009} we then know that $\extt{\Z[\frac{1}{3}]}{(\calE[2]}{\calE[2]}$
is generated by $\calE[4]$ and by extensions that are annihilated by $2$.
To prove the lemma, it remains to show that there are no non-trivial extensions of
$\calE[2]$ by itself that are annihilated by $2$.

To show this, we start by noting that
also $\Hom_{\Q_{2}}(\calE[2],\calE[2])$ is of order $2$ because
the prime $2$ does not split in $\Q(\calE[2])/\Q$.
In the long exact sequence of Proposition \ref{prop:mayervietoris},
the map 
$$\Hom_{\Q}(\calE[2],\calE[2]) \times \Hom_{ \Z_{2} }(\calE[2],\calE[2]) \longrightarrow
\Hom_{\Q_{2} }(\calE[2],\calE[2])
$$
is therefore surjective.
As a consequence, we obtain the inclusion: 
\begin{align*}
\extt{\Z[\frac{1}{3}]}{(\calE[2]}{\calE[2]} \longhookrightarrow \extt{\Z_{2}}^{1}{\calE[2]}{\calE[2]} \times \extt{\Z[\frac{1}{6}]}^{1}{\calE[2]}{\calE[2]} .
\end{align*}
By restricting the inclusion to the subgroup of extensions that are annihilated by $2$, we obtain:
\begin{align} \label{eqn:extincl}
\extt{\Z[\frac{1}{3}],[2]}^{1}{\calE[2]}{\calE[2]} \longhookrightarrow \extt{\Z_{2},[2]}^{1}{\calE[2]}{\calE[2]} \times
\extt{\Z[\frac{1}{6}],[2]}^{1}{\calE[2]}{\calE[2]} .
\end{align}
Now let $\gsc$ be an extension of $\calE[2]$ by $\calE[2]$ and 
suppose that $\gsc$ is annihilated by $2$.
So $\Q(\gsc)$ is contained in $\Q(\calE[4])$ by Lemma \ref{j027-lem:max2ext}
and the order of $\Gal(\Q(\gsc)/\Q(\calE[2]))$ must divide $2$.
As we saw in Lemma \ref{j027-lem:max2ext}, the only extensions of 
$\Q(\calE[2])$ that satisfy these conditions
are the ray class field of conductor $\bad_{2}^{4}\bad_{3}$
which is equal to $\Q(\calE[4])$,
or the ray class field of conductor $\bad_{2}^{4}$ which is equal to $\Q(\calE[2], i)$.
By \cite[Proposition 6.4]{Schoof:2003}, a result based on a calculation with finite Honda systems,
the number field $\Q(\gsc)$ is contained in a ray class field of $\Q(\calE[2])$ whose conductor
is such that the exponent at $\bad_2$ is at most $2$. 
Hence $\gsc_\Q$ is an $\F_2[S_3]$-module.
Again by Lemma \ref{j027-lem:f2s3_modules},
this module is semi-simple and we deduce that the sequence splits generically.

Next we consider the group scheme $\gsc_{\Z_2}$.
Both the scheme $\gsc_{\Z_{2}}$ and its Cartier dual are local group schemes:
They are extensions of the connected group scheme $\calE[2]_{\Z_2}$ by itself. 
The fact that the Cartier dual of $\gsc_{\Z_{2}}$ is such an extension follows from the self-duality of $\calE[2]$.
Hence $\gsc_{\Z_{2}}$ is biconnected and according to \cite[Section 3.3.5]{Raynaud:1974},
$\gsc_{\Z_{2}}$ is uniquely determined by its generic fiber that splits.
This implies that $\gsc_{\Z_{2}}$ also splits.

Now use inclusion (\ref{eqn:extincl}) above to deduce that all extensions annihilated by $2$
in fact split. The non-trivial extension that remains is $\calE[4]$.
\end{proof}

\subsection{Proof of Theorem \ref{j027-thm:main}} \label{j027-sec:proof}

\begin{proof}
Let $A$ be an abelian variety over $\Q$ with
good reduction outside $3$ and with level of ramification 
$\frac{1}{2}$ at $3$.
Let $\calA$ be the N\'eron model of $A$ over $\Z[\frac{1}{3}]$ and 
let $\calA[2]$ be the finite flat commutative $2$-torsion subgroup scheme.
The group scheme $\calA[2]$ is an object in $\catb$.
First we show that $\calA[2]$ does not have subquotients that are \'etale or of multiplicative type.
In order to deduce this from Theorem \ref{cat-thm:abvar_torsionfilter},
we check that the conditions in that theorem are satisfied for our category $\catb$.

By Proposition \ref{j027-prop:ext_mu2_z2},
the category $\catb$ satisfies 
the condition that for all simple non-\'etale group schemes $T$ in $\catb$ 
and all simple \'etale group schemes $\gsc_{\text{\'et}}$ in $\catb$,
the group $\extt{\Z[\frac{1}{3}]}^{1}{T}{\gsc_{\text{\'et}}}$ is trivial.
The second condition requires that the maximal abelian
extension of $\Q$, which is unramified outside $3$ and at most tamely ramified at $3$,
is cyclic. 
This maximal abelian extension is equal to $\Q(\zeta_3)$ and is cyclic over $\Q$. 
Hence $\calA[2]$ does not have subquotients that are \'etale or that are of multiplicative type.
The rest of the proof is analogous to the proof in Section \ref{sec:modularityj032}.
\end{proof}

\section{Good reduction outside $7$} \label{chap:j049}

Consider the newforms $f=49A$ and $g=49B$ of level $49$ and weight $2$
that occur in Table \ref{tab:modformlevel}.
We denote the abelian varieties over $\Q$ attached to the newforms $f$ and $g$ by $E$ and $B$
respectively.
The abelian variety $E$ is an elliptic curve and $B$ is a $2$-dimensional abelian variety
that is closely related to $E$:
It is isogenous to $E \times E$ over the field $\Q(\zeta_7)^+ := \Q(\zeta_7 + \zeta_7^{-1})$.
The aim of this chapter is to prove the following part of Theorem \ref{maintheorem}:

\begin{theorem} \label{j049-thm:main}
Let $A$ be an abelian variety over $\Q$ with good reduction outside $7$
such that $A$ obtains semi-stable reduction over an at most tamely ramified 
extension of $\Q$.
Then, under assumption of the generalized Riemann hypothesis, 
$A$ is isogenous to a product of copies of $E$ and a product
of copies of $B$.
\end{theorem}

In order to prove Theorem \ref{j049-thm:main},
we work with the following category:

\begin{definition}
Let $\catf$ be the category whose objects are finite flat commutative group schemes $\gsc$
of $2$-power order over $\Z[\frac{1}{7}]$ such that $\Q(\gsc)/\Q$ is at most tamely ramified at $7$
and whose morphisms are group scheme morphisms over $\Z[\frac{1}{7}]$.
\end{definition}

Let $\calE$ and $\calB$ denote the N\'eron models of $E$ and $B$ respectively.
If we adjoin the $2$-torsion points of $E$ to $\Q$, we obtain the number field $\Q(\sqrt{-7})$.
Therefore the subgroup scheme $\calE[2]$ of $2$-torsion points of $\calE$ is an object
in $\catf$.
Since $B$ becomes isogenous to $E \times E$ over $\Q(\zeta_7)^+$,
and since $\Q(\zeta_7)^+/\Q$ is tamely ramified at $7$, 
also the group scheme $\calB[2]$ is an object in $\catf$.

\subsection{Simple group schemes} \label{j049-sec:simple}

We classify the simple group schemes in the category $\catf$.
We do this by using the maximal $2$-torsion extension $\maxextcatc$
of $\catf$, that is, the field generated by all $\Alg{\Q}$-points
of group schemes in $\catf$ annihilated by $2$.
The field $\maxextcatc$ is contained in the field $L$ occurring in the next lemma.

\begin{lemma} \label{j049-lem:simple_groupschemes}
The maximal extension $L$ of $\Q$ that is at most tamely ramified at $7$
and ramified of level $1$ at $2$, 
is a degree $2$-power extension of $\Q(\zeta_{28})$.
\end{lemma}
\begin{proof}
The root discriminant $\delta_L$ of the extension $L$ is strictly smaller than $28$.
Using the tables of \cite{Martinet:1982} under the assumption of the generalized Riemann hypothesis,
this implies that $L$ is a finite extension over $\Q$ and has degree smaller than $725$.
The $\Alg{\Q}$-points of the extension 
 of $\Z/2\Z$ by $\mu_2$, represented by the scheme 
$\Spec( \prod_{i=0}^{i=1} \Z[\frac{1}{7}] [X_i] /(X_{i}^2 - (-7)^{i})  )$ over $\Z[\frac{1}{7}]$,
generate the extension $\Q(\sqrt{-7})/\Q$.
Another group scheme in $\catf$ that is annihilated by $2$ is 
the permutation group scheme:
Over $\Z[\zeta_7/7]$ it becomes isomorphic to $(\Z/2\Z)^6$, but over $\Z[\frac{1}{7}]$
the cyclic Galois group $\Gal(\Q(\zeta_7)/\Q)$ permutes the factors.
Together with the non-trivial extension 
$T[2,-1]$ of $\Z/2\Z$ by $\mu_2$ over $\Z[\frac{1}{7}]$,
we have the following extensions of number fields:
$$
\Q \subset_{12} \Q(i, \zeta_7) = \Q(\zeta_{28}) \subset_{\leq 60} L .
$$
The discriminant of $\Q(\zeta_{28})$ is $2^{12} 7^{10}$.
By the Kronecker-Weber Theorem and the bound on the ramification at the primes
$2$ and $7$ in $L/\Q$,
the extension $\Q(\zeta_{28})/\Q$ is the maximal abelian extension inside $L/\Q$.

\vspace{1ex}
\noindent
{\bf Claim 1} :
The Galois group $\Gal(L/\Q(\zeta_{28}))$ is solvable.
\begin{proof}
We proceed by contradiction 
and suppose that the Galois group $\Gal(L/\Q(\zeta_{28})$ is not solvable.
Then $\Gal(L/\Q(\zeta_{28})$ is isomorphic to the alternating group $A_5$
of order $60$.

First suppose that the extension $L/\Q(\zeta_{28})$ is unramified at the primes 
lying over $2$ and $7$.
Then the root discriminant of $L$ is equal to the root discriminant
of $\Q(\zeta_{28})$.
According to the tables in \cite{Martinet:1982},
the extension $L/\Q$ cannot have degree $720$.
On the other hand, the extension $L/\Q(\zeta_{28})$ cannot be ramified at $7$:
If $L/\Q(\zeta_{28})$ is ramified at $7$, the inertia subgroup at the prime lying over $7$
of $\Gal(L/\Q(\zeta_{28}))$ is cyclic.
We note that any decomposition group is solvable and the inertia subgroup
in a decomposition group is normal.
All solvable subgroups of $A_5$ have order less or equal to $12$.
The order of any normal cyclic subgroup in any of these solvable subgroups
is at most $5$.
Therefore, the inertia subgroup of a prime lying above $7$ inside the Galois
extension $L/\Q$ has order at most $30$.
This means that $\delta_L < 4 \cdot 7^{29/30}$.
Again using the tables in \cite{Martinet:1982} we see that $[L : \Q]=720$ is impossible.
Finally, the extension $L/\Q(\zeta_{28})$ cannot be ramified at a prime lying above $2$ either:
The prime $2$ splits in $\Q(\zeta_{28})/\Q$ into two primes $\pi$ and $\pibar$.
We use again that all solvable subgroups of $A_5$ have order at most $12$.
We may assume that such a subgroup corresponds to the inertia group at $\pi$ or $\pibar$.
First we argue that these subgroups must be abelian.
The non-abelian subgroups of $A_5$ of order at most $12$ are isomorphic to
either the dihedral groups $D_3, D_5$ or the alternating group $A_4$.
Because they correspond to a ramification group at $\pi$ or $\pibar$, the wild inertia subgroup is normal. But
$D_3, D_5$ or $A_4$ do not admit normal subgroups of order resp. $2$, $2$ or $4$.
Hence we assume that the ramification group at $\pi$ or $\pibar$ is abelian,
which implies that it is of order at most $5$.

Since the the level of ramification at $2$ of $L/\Q$ is at most $1$,
the exponent of the Artin conductor of 
any character of the extension corresponding to the inertia subgroup at $\pi$ or $\pibar$
is at most $2$.
Suppose that the inertia subgroup at $\pi$ of $L/\Q(\zeta_{28})$ has order $4$.
We only have wild ramification at $\pi$ (resp. $\pibar$)
inside $L/\Q(\zeta_{28})$: 
We count one trivial character and
three characters of conductor $\pi^2$ (resp. $\pibar^2$).
Then using the conductor-discriminant formula
we compute that the $2$-adic valuation of the root discriminant of $L$ is at most $\frac{14}{8}$.
Hence $\delta_L$ is at most $7^{5/6} \cdot 2^{14/8}$.
The tables in \cite{Martinet:1982} then imply that $[L : \Q]=720$ is impossible.
The same argument works if the order of the inertia subgroup of a prime lying above $2$
in $L/\Q$ has order $1,2,3$ or $5$; we would only get smaller root discriminants for $L/\Q$.
\end{proof}

\noindent
{\bf Claim 2} :
The Galois group $\Gal(L/\Q(\zeta_{28}))$ is a $2$-group.
\begin{proof}

Let $\pi_{7}, \pi$ and $\overline{\pi}$ be primes in the ring of integers
of $\Q(\zeta_{28})$ such that $\pi_7$ lies over $7$, and such that $\pi$ and $\pibar$ are distinct and
lie over $2$.
The ray class field $R$ of conductor $\pi_{7} \overline{\pi}^2 \pi^2$
is an extension of degree $16$ of $\Q(\zeta_{28})$ with root discriminant  
$7^{11/12} \cdot 2^{30/16}$ and that the level of ramification at primes over $2$ is $1$.
It is the maximal abelian extension of $\Q(\zeta_{28})$ inside $L/\Q(\zeta_{28})$.
It suffices to show that there is no extension of degree $3$ of $R$ inside $L$.
To do this we apply \cite[Corollary 3.2]{Schoof:2003} with the groups $\Gamma=\Gal(L/\Q)$
and $\Gamma'=\Gal(L/\Q(\zeta_{28}))$:
Given a finite group $\Gamma$ such that $\Gamma' / \Gamma''$ is a $2$-group and
$\# \Gamma'' < 25$,
then either $9$ divides $\# \Gamma''$ or $\Gamma''$ is a $2$-group.
\end{proof}

\end{proof}

Let $\chi$ be one of the two non-trivial characters $\Gal(\Q(\zeta_7)^+/\Q) \rightarrow \F_4^*$
of conductor $7$ and order $3$. 
If we want the coefficients of $\chi$ to be in $\F_2$ instead of $\F_4$, 
then $\chi$ becomes a $2$-dimensional representation with image generated by
$
M=\left(\begin{array}{cc}1 & 1\\1 & 0\end{array}\right)$ .
Similarly, the representation $\chi^{-1}$ with coefficients in $\F_2$ gives a $2$-dimensional representation.

Let $C(\chi)$ denote the twist of the group scheme 
$(\Z/2\Z \times \Z/2\Z)$ by the character $\chi$.
The scheme $C(\chi)$ is a prolongation of the representation $\chi$ to a finite flat commutative group scheme over $\Z[\frac{1}{7}]$.
This group scheme is flat over $\Z[\frac{1}{7}]$ since $\chi$ is only ramified at $7$.
The Cartier dual of $C(\chi)$ is denoted by $C(\chi)^*$ and is a twist of $(\mu_2 \times \mu_2)$ by $\chi$.

\begin{proposition} \label{j049-prop:simple_groupschemes}
Any simple group scheme in $\catf$ is isomorphic to either
$\mu_2$, $\Z/2\Z$, $C(\chi)$, $C(\chi)^*$, $C(\chi^2)$ or $C(\chi^2)^*$.
\end{proposition}
\begin{proof}
Any $p$-power order normal subgroup of a Galois group of a finite extension of $\Q$
acts trivial on the generic fiber of a simple $p$-power order finite flat commutative group
scheme over a subring of $\Q$.
Therefore, by Lemma \ref{j049-lem:simple_groupschemes}, 
the Galois action on the generic fiber of a simple group scheme in $\catf$ factors through $\Gal(\Q(\zeta_7)^+/\Q)$.
It follows that the generic fibers of the simple group schemes of $\catf$ are $\F_2[C_3]$-modules,
where $C_3$ denotes the cyclic group of order $3$.
Hence they are of order $2$ (trivial character) or of order $4$ (the representation $\chi \oplus \chi^{-1}$).

To classify group schemes of order $2$, we use \cite[Theorem 2]{TateOort:1970}:
All group schemes in $\catf$ of order $2$ are isomorphic to either $\mu_2$ or $\Z/2\Z$.

To classify group schemes of order $4$ we proceed as follows:
Let $\gsc$ be a simple finite flat group scheme over $\Z[\frac{1}{7}]$ of order $4$.
We first prove the following claim:

\vskip 10pt
\noindent
{\bf Claim }: The group scheme $\gsc_{\Z_2}$ is isomorphic to one of the group schemes 
$C(\chi)_{\Z_2}$, $C(\chi)_{\Z_2}^*$, $C(\chi^2)_{\Z_2}$ or $C(\chi^2)_{\Z_2}^*$.
\begin{proof}
Since the prime $2$ is inert in $\Q(\zeta_7)^+/\Q$, the action of $\Gal(\Alg{\Q}_2/\Q_2)$ 
on $\gsc(\Alg{\Q}_2)$ is also irreducible.
This means that $\gsc_{\Z_2}$ is a simple group scheme over $\Z_2$
and that $\gsc$ is local, i.e., connected over the ring of integers of an algebraic closure of $\Q_2$, or \'etale.
Also the Cartier dual of $\gsc_{\Z_2}$ is a simple group scheme over $\Z_2$ and either local or \'etale.

\begin{itemize}
\item
If $\gsc_{\Z_2}$ is an \'etale-local group scheme, which means that $\gsc_{\Z_2}$ is \'etale and its Cartier dual
$\gsc_{\Z_2} ^*$ is local, then by Proposition \ref{j049-prop:simple_groupschemes}
the group scheme $\gsc_{\Z_2}$ is isomorphic to either $C(\chi)_{\Z_2}$ or $C(\chi^2)_{\Z_2}$.

\item
If $\gsc_{\Z_2}$ is a local-\'etale group scheme, then $\gsc_{\Z_2}$ is isomorphic to $C(\chi)_{\Z_2}^*$ or $C(\chi^2)_{\Z_2}^*$.

\item
The group scheme $\gsc_{\Z_2}$ cannot be \'etale-\'etale since $\gsc_{\Z_2}$ has order a power of $2$.

\item
Finally, the group scheme $\gsc_{\Z_2}$ cannot be local-local: 
If $\gsc_{\Z_2}$ is local-local, then by \cite[Proposition 6.1]{Schoof:2003} the extension $\Q_2(\gsc(\Alg{\Q}_2))/\Q_2$ is
a degree $6$ extension.
However, as we noticed before, the extension $\Q_2(\gsc(\Alg{\Q}_2))/\Q_2$ is a degree $3$ extension.
\end{itemize}

\end{proof}

From the claim it follows that $\gsc$ is \'etale or connected.
If $\gsc$ is \'etale, then $\gsc$ is isomorphic to either $C(\chi)$ or $C(\chi^2)$, and we are done.
If $\gsc$ is connected, then we can twist the group scheme $\gsc$ by a power of $\chi$ such that the Galois action on the generic fiber is trivial.
This twisted group scheme, say $\gsc'$, is not simple anymore and an extension of $\mu_2$ by $\mu_2$.
The Cartier dual of $\gsc'$ is therefore \'etale and isomorphic to either $C(\chi)$ or $C(\chi^2)$.
It follows that $\gsc$ is isomorphic to one of the group schemes $C(\chi)^*$ or $C(\chi^2)^*$.
\end{proof}

\subsection{Extensions of simple group schemes}

We calculate extension groups of 
some of the simple group schemes in the category $\catf$.
However, we perform these calculations not in the category $\catf$ but
in the following category $\cate$:

\begin{definition}
Let $K=\Q(\zeta_7)$ and let $\pi_7$ be the unique prime
lying above $7$ in $O_K$.
Let $\cate$ be the category of finite flat commutative group schemes $\gsc$
of $2$-power rank over $O_K[\frac{1}{7}]$ such that $K(\gsc)/K$ is at most
tamely ramified at $\pi_7$. 
\end{definition}

Every object in $\catf$ becomes an object in $\cate$ by base change.
Besides the group scheme $\calE[2]$,
there are other group schemes in the category $\cate$ associated to $E$
that we will now describe.
The prime $2$ splits into two primes $\pi$ and $\pibar$ in the extension $\Q(\sqrt{-7})/\Q$.
Since $K/\Q(\sqrt{-7})$ is inert at $\pi$ and $\pibar$, 
also in $K$ the ideals $\pi$ and $\pibar$ are prime ideals.
The elliptic curve $E$ has complex multiplication over the field $\Q(\sqrt{-7})$.
So the $\pi$- and $\pibar$-torsion of $\calE$
are finite flat commutative group schemes of $2$-power order over the ring $O_K[\frac{1}{7}]$.
They are objects in the category $\cate$ and we denote them by $\calE[\pi]$ and $\calE[\pibar]$ respectively.
An explicit description of $\calE[\pi]$ and $\calE[\pibar]$ is given as follows: 
\begin{align*}
\calE[\pi] & = \Spec(O_K[\frac{1}{7}][X]/X^2 - \pi X ) \quad \text{with} \quad \Delta_{\calE[\pi]} : X \mapsto X \otimes 1 + 1 \otimes X - \pibar X \otimes X , \\
\calE[\pibar] & = \Spec(O_K[\frac{1}{7}][X]/X^2 - \pi X) \quad \text{with} \quad \Delta_{\calE[\pibar]} : X \mapsto X \otimes 1 + 1 \otimes X - \pi X \otimes X 
\end{align*}
where $\Delta_{\calE[\pi]}$ and $\Delta_{\calE[\pibar]}$ denote the comultiplication maps.
The counit maps send $X$ to zero.
We will not discuss the coinverse maps.
The group schemes $\calE[\pi]$ and $\calE[\pibar]$ are each other's Cartier dual.
Note that over the completions at the primes above $2$, 
we have the following isomorphisms:
$$
\calE[\pi]_{O_{K_{\pi}}} \simeq \mu_2 , \quad \calE[\pibar]_{O_{K_{\pibar}}} \simeq \mu_2 , \quad 
\calE[\pi]_{O_{K_{\pibar}}} \simeq \Z/2\Z \quad  \text{and} \quad \calE[\pibar]_{O_{K_{\pi}}} \simeq \Z/2\Z .
$$

The group scheme $\calE[\pi] \times \calE[\pibar]$ is an extension of $\mu_2$ by $\Z/2\Z$ over $O_K[\frac{1}{7}]$.
If we write 
$$
\calE[\pi] \times \calE[\pibar] = \Spec \left( O_K[\frac{1}{7}][X,Y]/(X^2 - \pi X , Y^2 - \pibar Y) \right)
$$
and $\Z/2\Z = \Spec(O_K[\frac{1}{7}][Z]/(Z^2 - Z))$, then the inclusion morphism of $\Z/2\Z$ into
$\calE[\pi] \times \calE[\pibar]$ is given by the algebra maps 
$X \mapsto \pi Z$ and $Y \mapsto \pibar Z$.
See also \cite[p. 428]{Schoof:2003}.

\begin{lemma} \label{j049-lem:ext_mu2_Z2_Gpi}
The group $\extt{O_K[\frac{1}{7}]}{\mu_2}{\Z/2\Z}$ is generated by the extension $\calE[\pi] \times \calE[\pibar]$ and
has order $2$.
\end{lemma}
\begin{proof}
We obtain the following exact sequence from the long exact sequence of Proposition \ref{prop:mayervietoris}:
\begin{align*}
0 & \longrightarrow  \Hom_{O_K[\frac{1}{7}]}(\mu_2,\Z/2\Z) \longrightarrow 
\Hom_{O_K[\frac{1}{7}] \otimes \Z_2} (\mu_2,\Z/2\Z) \times \Hom_{K}(\mu_2,\Z/2\Z) \\
& \quad\quad \longrightarrow \Hom_{K \otimes \Q_2}(\mu_2,\Z/2\Z) \longrightarrow \extt{O_K[\frac{1}{7}]}{\mu_2}{\Z/2\Z} \\
& \qquad\qquad \longrightarrow \extt{K}{\mu_2}{\Z/2\Z} \times \extt{O_K[\frac{1}{7}] \otimes \Z_2}{\mu_2}{\Z/2\Z}.
\end{align*}
The groups $\Hom_{O_K[\frac{1}{7}]} (\mu_2,\Z/2\Z)$ and $\Hom_{O_K[\frac{1}{7}] \otimes \Z_2} (\mu_2,\Z/2\Z)$
are trivial, $\Hom_{K}(\mu_2,\Z/2\Z)$ has order $2$,
$\Hom_{K \otimes \Q_2}(\mu_2,\Z/2\Z)$ has order $4$
and $\extt{O_K[\frac{1}{7}] \otimes \Z_2}{\mu_2}{\Z/2\Z}$
is trivial since every extension is split by its connected component.
It suffices to prove that the image of $\extt{O_K[\frac{1}{7}]}{\mu_2}{\Z/2\Z}$ in $\extt{K}{\mu_2}{\Z/2\Z}$ is trivial:
It then follows that $\extt{O_K[\frac{1}{7}]}{\mu_2}{\Z/2\Z}$ has order $2$,
and $\calE[\pi] \times \calE[\pibar]$ is a non-trivial extension.

Let $\gsc$ be an extension of $\mu_2$ by $\Z/2\Z$ over $O_K[\frac{1}{7}]$.
Since the group $\extt{O_K[\frac{1}{7}] \otimes \Z_2}{\mu_2}{\Z/2\Z}$ is trivial, 
the group scheme $\gsc$ is annihilated by $2$ and 
the extension
$K(\gsc)/K$ is unramified at both the primes $\pi$ and $\pibar$.
One easily shows that the Galois extension $K(\gsc)/K$ is at most of degree $2$,
hence we can use class field theory.

The extension $K(\gsc)/K$ is at most tamely ramified at the prime lying above $7$.
A calculation (see Appendix) shows that 
the ray class field of conductor $\pi_7$ is trivial.
Therefore, $K(J)/K$ is trivial and we conclude that the image of $\extt{O_K[\frac{1}{7}]}{\mu_2}{\Z/2\Z}$ in $\extt{K}{\mu_2}{\Z/2\Z}$
is trivial.
This concludes the proof of the lemma.
\end{proof}

\begin{lemma} \label{lem:j049_ext_Gpi_Gpibar}
The groups $\extt{O_K[\frac{1}{7}]}{\calE[\pi]}{\calE[\pibar]}$ and $\extt{O_K[\frac{1}{7}]}{\calE[\pibar]}{\calE[\pi]}$ are trivial.
\end{lemma}
\begin{proof}
It suffices to prove that $\extt{O_K[\frac{1}{7}]}{\calE[\pibar]}{\calE[\pi]}$ is trivial. 
By Cartier duality, it then
follows that the group $\extt{O_K[\frac{1}{7}]}{\calE[\pi]}{\calE[\pibar]}$ is also trivial.

By applying Proposition \ref{prop:mayervietoris} with $\gsc' = \calE[\pibar]$
and $\gsc''=\calE[\pi]$, we obtain the long exact sequence
\begin{align*}
0 \longrightarrow & \Hom_{O_K[\frac{1}{7}]}(\calE[\pibar],\calE[\pi]) \longrightarrow \Hom_{O_K[\frac{1}{7}] \otimes \Z_2} (\calE[\pibar],\calE[\pi]) \times \Hom_{K}(\calE[\pibar],\calE[\pi]) \\
& \quad \longrightarrow \Hom_{K \otimes \Q_2}(\calE[\pibar],\calE[\pi]) \longrightarrow \extt{O_K[\frac{1}{7}]}{\calE[\pibar]}{\calE[\pi]} \\
& \qquad \longrightarrow \extt{K}{\calE[\pibar]}{\calE[\pi]} \times \extt{O_K[\frac{1}{7}] \otimes \Z_2}{\calE[\pibar]}{\calE[\pi]}.
\end{align*}
From this, together with the fact that
$\extt{O_K[\frac{1}{7}],[2]}{\calE[\pibar]}{\calE[\pi]} = \extt{O_K[\frac{1}{7}]}{\calE[\pibar]}{\calE[\pi]}$,
we deduce the injectivity of the map
\begin{align*}
\extt{O_K[\frac{1}{7}]}{\calE[\pibar]}{\calE[\pi]} \longhookrightarrow \extt{K,[2]}{\calE[\pibar]}{\calE[\pi]}
\times \extt{O_K[\frac{1}{7}] \otimes \Z_2, [2]}{\calE[\pibar]}{\calE[\pi]} .
\end{align*}
The image in $\extt{K,[2]}{\calE[\pibar]}{\calE[\pi]}$ is trivial
by the following argument:
Any extension $\gsc$ of $\calE[\pibar]$  by $\calE[\pi]$ over $O_K[\frac{1}{7}]$
generates an extension $K(\gsc)/K$.
The extension $K(\gsc)/K$ sits inside the ray class field of $K$ of conductor $\pi_7 \pi^2$:
First of all, there is no ramification at $\pibar$ since $\gsc_{O_{K_{\pibar}}}$
splits by the connected component.
Secondly, the $\pi^2$ in the conductor comes from the fact that 
$\gsc_{O_{K_{\pi}}}$ is isomorphic to an extension of $\Z/2\Z$ by $\mu_2$ annihilated by $2$,
whose conductor at $\pi$ is at most $2$; cf. 
\cite[Proposition 8.10.5, p. 263]{KatzMazur:1985}.

Another calculation (see Appendix) shows that the ray class field 
of conductor $\pi_7 \pi^2$ is of degree $2$ over $K$.
Since $\gsc_{O_{K_{\pibar}}}$ splits, the prime $\pibar$ splits in the number field extension $K(\gsc)/K$.
However, $\pibar$ does not split in the ray class field of $K$ of conductor $\pi_7 \pi^2$.
Therefore, the extension $\gsc$ is trivial over $K$.

Similarly, the image of $\extt{O_K[\frac{1}{7}]}{\calE[\pibar]}{\calE[\pi]}$ in $\extt{O_K \otimes \Z_2}{\calE[\pibar]}{\calE[\pi]}$ is trivial: 
We already saw that the extension $\gsc$ splits locally at $\pibar$.
Locally at $\pi$ we get an extension again as in 
\cite[Proposition 8.10.5, p. 263]{KatzMazur:1985}.
with trivial Galois action. 
Hence, $\gsc_{O_{K_{\pi}}}$ is also trivial.
\end{proof}

Because of Lemma \ref{lem:j049_ext_Gpi_Gpibar} and as $\calE[4]$ is a non-trivial extension of 
$\calE[2]=\calE[\pi] \times \calE[\pibar]$ by $\calE[2]$,
we know that $\extt{O_K[\frac{1}{7}]}{\calE[\pi]}{\calE[\pi]}$ and $\extt{O_K[\frac{1}{7}]}{\calE[\pibar]}{\calE[\pibar]}$ are both at least of order $2$.
The next lemma says that these extension groups are exactly of order $2$.

\begin{lemma} \label{j049-lem:ext_gpi_by_gpi}
The groups $\extt{O_K[\frac{1}{7}]}{\calE[\pi]}{\calE[\pi]}$ and $\extt{O_K[\frac{1}{7}]}{\calE[\pibar]}{\calE[\pibar]}$ are both of order $2$.
\end{lemma}
\begin{proof}
We only prove the lemma for $\extt{O_K[\frac{1}{7}]}{\calE[\pi]}{\calE[\pi]}$ as the statement for
$\extt{O_K[\frac{1}{7}]}{\calE[\pibar]}{\calE[\pibar]}$ follows by Cartier duality.
Denote by $\extt{O_K[\frac{1}{7}],[2]}{\calE[\pi]}{\calE[\pi]}$ the subgroup of 
extensions of $\calE[\pi]$ by $\calE[\pi])$ that are killed by $2$.
Then, by \cite[Lemma 2.1]{Schoof:2009}, the following sequence is exact:
\begin{align*}
0 \longrightarrow \extt{O_K[\frac{1}{7}],[2]}{\calE[\pi]}{\calE[\pi]} \longrightarrow \extt{O_K[\frac{1}{7}]}{\calE[\pi]}{\calE[\pi]} \longrightarrow 
\Hom_{ G_K }(\calE[\pi],\calE[\pi])  .
\end{align*}
The group $\Hom_{ G_K }(\calE[\pi],\calE[\pi])$ is of order $2$.
Since we know that $\extt{O_K[\frac{1}{7}]}{\calE[\pi]}{\calE[\pi]}$ is at least of order $2$,
to prove the lemma it suffices to show that $\extt{O_K,[2]}{\calE[\pi]}{\calE[\pi]}$ is trivial.
We again use Proposition \ref{prop:mayervietoris} to obtain the long exact sequence
\begin{align*}
0 & \longrightarrow \Hom_{O_K[\frac{1}{7}]}(\calE[\pi],\calE[\pi]) \longrightarrow 
\Hom_{O_K[\frac{1}{7}] \otimes \Z_2} (\calE[\pi],\calE[\pi]) \times \Hom_{K}(\calE[\pi],\calE[\pi]) \\
& \quad \longrightarrow \Hom_{K \otimes \Q_2}(\calE[\pi],\calE[\pi]) \longrightarrow \extt{O_K[\frac{1}{7}]}{\calE[\pi]}{\calE[\pi]} \\
& \qquad \longrightarrow \extt{K}{\calE[\pi]}{\calE[\pi]} 
\times \extt{O_K[\frac{1}{7}] \otimes \Z_2}{\calE[\pi]}{\calE[\pi]} .
\end{align*}
We deduce from this the injectivity of the following map:
\begin{align} \label{j049-eqn:map:ext_gpi_by_gpi}
\extt{O_K[\frac{1}{7}],[2]}{\calE[\pi]}{\calE[\pi]} \longhookrightarrow 
\extt{K,[2]}{\calE[\pi]}{\calE[\pi]} \times \extt{O_K[\frac{1}{7}] \otimes \Z_2,[2]}{calE[\pi]}{\calE[\pi]} .
\end{align}
We prove that the image of $\extt{O_K[\frac{1}{7}],[2]}{\calE[\pi]}{\calE[\pi]}$ in $\extt{K,[2]}{\calE[\pi]}{\calE[\pi]}$ is trivial.
From this, we will then also deduce that the image of $\extt{O_K[\frac{1}{7}],[2]}{\calE[\pi]}{\calE[\pi]}$ in 
$\extt{O_K[\frac{1}{7}] \otimes \Z_2,[2]}{calE[\pi]}{\calE[\pi]}$ is trivial.

So let $\gsc$ be an extension of $\calE[\pi]$ by $\calE[\pi]$ over $O_K[\frac{1}{7}]$ killed by $2$.
The extension $K(\gsc)/K$ is only ramified at the prime above $7$ and at the prime $\pi$.
The group scheme $\gsc_{O_{K_\pi}}$ is an extension of $\mu_2$ by itself.
The Cartier dual $\gsc_{O_{K_\pi}}^*$ of $\gsc_{O_{K_\pi}}$ is an extension of $\Z/2\Z$ by itself.
Since $\gsc_{O_{K_\pi}}^*$ is annihilated by $2$ and is \'etale, $\gsc_{O_{K_\pi}}^*$ is split over an unramified
extension of $K_\pi$.
This means that also $\gsc_{O_{K_\pi}}$ is split over an unramified extension
of $K_\pi$.
We deduce that $K(\gsc)/K$ is only tamely ramified at the prime above $7$
and is abelian, 
since the group $\Gal( K(\gsc)/K )$ has order dividing $2$.
However, we saw that the ray class field of conductor $\pi_7$ is trivial.
Hence, $\gsc$ is generically a split extension. 

It remains to show that the image
$\extt{O_K[\frac{1}{7}],[2]}{\calE[\pi]}{\calE[\pi]}$ in
$\extt{O_K[\frac{1}{7}] \otimes \Z_2,[2]}{calE[\pi]}{\calE[\pi]}$ is trivial.
The group scheme $\gsc_{O_{K_\pi}}$ 
is a group scheme of multiplicative type
and the group scheme $\gsc_{O_{K_\pibar}}$ is \'etale.
Since the Galois action on $\gsc(\Alg{K})$ is trivial,
both the extensions $\gsc_{O_{K_\pi}}$ and $\gsc_{O_{K_\pibar}}$ are trivial.
From the injectivity of the map (\ref{j049-eqn:map:ext_gpi_by_gpi}), we see that $\extt{O_K[\frac{1}{7}],[2]}{\calE[\pi]}{\calE[\pi]}$ is trivial.
\end{proof}

\begin{lemma} \label{j049-lem:ext_mu2_Epi}
The extension groups 
$$
\extt{O_K[\frac{1}{7}]}{\mu_2}{\calE[\pi]}, \,
\extt{O_K[\frac{1}{7}]}{\mu_2}{\calE[\pibar]}, \, 
\extt{O_K[\frac{1}{7}]}{\calE[\pi]}{\Z/2\Z}, \,
\extt{O_K[\frac{1}{7}]}{\calE[\pibar]}{\Z/2\Z}
$$
are trivial.
\end{lemma}
\begin{proof}
By Cartier duality, it suffices to prove that
\extt{O_K[\frac{1}{7}]}{\calE[\pi]}{\Z/2\Z}
and $\extt{O_K[\frac{1}{7}]}{\calE[\pibar]}{\Z/2\Z}$
are trivial.
We only prove that $\extt{O_K[\frac{1}{7}]}{\calE[\pi]}{\Z/2\Z}$ is trivial: 
The exact same argument given below with $\pi$ replaced by $\pibar$ shows
that $\extt{O_K[\frac{1}{7}]}{\calE[\pibar]}{\Z/2\Z}$ is trivial.
From Proposition \ref{prop:mayervietoris},
we obtain the following exact sequence:
\begin{align*}
0 & \longrightarrow \Hom_{O_K[\frac{1}{7}]}(\calE[\pi],\Z/2\Z) \longrightarrow
 \Hom_{O_K[\frac{1}{7}] \otimes \Z_2} (\calE[\pi],\Z/2\Z) \times \Hom_{K}(\calE[\pi],\Z/2\Z)  \\
& \quad \longrightarrow \Hom_{K \otimes \Q_2}(\calE[\pi],\Z/2\Z) \longrightarrow \extt{O_K[\frac{1}{7}]}{\calE[\pi]}{\Z/2\Z} \\
& \qquad \longrightarrow \extt{K}{\calE[\pi]}{\Z/2\Z} \times \extt{O_K[\frac{1}{7}] \otimes \Z_2}{\calE[\pi]}{\Z/2\Z}.
\end{align*}
From this long exact sequence we deduce that
\begin{align} \label{eqn:incl_Gpi_z2z}
\extt{O_K[\frac{1}{7}]}{\calE[\pi]}{\Z/2\Z} \longhookrightarrow \extt{K}{\calE[\pi]}{\Z/2\Z} \times \extt{O_K[\frac{1}{7}] \otimes \Z_2}{\calE[\pi]}{\Z/2\Z}
\end{align}
is injective.
We prove that the image of $\extt{O_K[\frac{1}{7}]}{\calE[\pi]}{\Z/2\Z}$ in $\extt{K}{\calE[\pi]}{\Z/2\Z}$ is trivial.
Any extension $\gsc$ in $\extt{O_K[\frac{1}{7}]}{\calE[\pi]}{\Z/2\Z}$ is annihilated by $2$ as
$\gsc_{O_{K_\pi}}$ is split.
Moreover, $\gsc_{O_{K_\pibar}}$ is an \'etale extension.
Hence $K(\gsc)/K$ is unramified outside $\pi_7$.
As we have seen before, there is no abelian extension of $K$ that is only tamely ramified at $\pi_7$.
The image of 
$\extt{O_K[\frac{1}{7}]}{\calE[\pi]}{\Z/2\Z}$
in $\extt{O_K[\frac{1}{7}] \otimes \Z_2}{\calE[\pi]}{\Z/2\Z}$
is also trivial since the group scheme 
$\gsc_{O_{K_\pibar}}$ is \'etale, annihilated by $2$ and has trivial Galois action.
We conclude from (\ref{eqn:incl_Gpi_z2z}) that  $\extt{O_K[\frac{1}{7}]}{\calE[\pi]}{\Z/2\Z}$ is trivial.
\end{proof}

\subsection{Proving that certain abelian varieties are modular}

We prove Theorem \ref{j049-thm:main}.
Contrary to the case of good reduction outside $3$, we cannot apply Theorem \ref{cat-thm:abvar_torsionfilter}.
Consider an abelian variety $A$ over $\Q$ 
with good reduction outside $7$ that obtains semi-stable reduction
over an at most tamely ramified extension at $7$ over $\Q$.
Let $\calA$ denote the N\'eron model of $A$.
The $2^n$-torsion group schemes $\calA[2^n]$ 
are objects in the category $\catf$.
It follows by Proposition \ref{j049-prop:simple_groupschemes}, Lemmas \ref{j049-lem:ext_mu2_Z2_Gpi},
\ref{lem:j049_ext_Gpi_Gpibar} and \ref{j049-lem:ext_mu2_Epi},
that $\calA[2^n]$ as an object in the category $\cate$ can be filtered as follows
$$
0 \subset M(n) \subset \gsc'(n) \subset \gsc''(n) \subset \calA[2^n] ,
$$
where $M(n), \gsc'(n)$ and $\gsc''(n)$ are closed flat subgroup schemes of 
$\calA[2^n]$ such that $M(n)$ is filtered by copies of $\mu_2$,
the quotient $\gsc'(n)/M(n)$ is filtered by copies of $\calE[\pi]$,
the quotient $\gsc''(n)/\gsc'(n)$ is filtered by copies of $\calE[\pibar]$
and $\calA[2^n]/\gsc''(n)$ is filtered by copies of $\Z/2\Z$.

We are now going to prove that any $2$-power order group scheme over $O_K[\frac{1}{7}]$
that can be filtered by copies of $\calE[\pi]$ and $\calE[\pibar]$, is isomorphic to 
a group scheme of the form $\calE[2^n]$.
To prove this, we first generalize \cite[Corollary 8.2]{Schoof:2005} to the following theorem:

\begin{theorem} \label{j049-thm:grpschm_filtration} 
Let $O$ be a Noetherian domain of characteristic $0$ and let
$\catc$ be a full subcategory of the category of finite flat group schemes of $p$-power order
over $O$.
Suppose that $\catc$ is closed under taking products, closed flat subgroup schemes and quotients by closed
flat subgroup schemes.
Let $\{ G_i \}_{i \in I}$ be a finite set of $p$-divisible groups over $O$ such that
each $G_i$ satisfies:
\begin{itemize}
\item the endomorphism ring $R_i := \End(G_i)$ is a discrete valuation ring with uniformizer $\pi_i$ and residue field $k=R_i /\pi_i R_i$

\item every group scheme $G_i[\pi_i^n]$ is an object in the category $\catc$

\item the map
$$
\Hom_O(G_i[\pi_i],G_i[\pi_i]) \xrightarrow{\delta} \extt{\catc}{G_i[\pi_i]}{G_i[\pi_i]}
$$
associated to the exact sequence $0 \rightarrow G_i[\pi_i] \rightarrow G_i[\pi_i^2] \rightarrow G_i[\pi_i] \rightarrow 0$
is an isomorphism of $1$-dimensional $k$-vector spaces

\item for each $i,j \in I$ with $i \not= j$, the groups $\extt{\catc}{G_i[\pi_i]}{G_j[\pi_j]}$ and $\extt{\catc}{G_j[\pi_j]}{G_i[\pi_i]}$
are trivial .
\end{itemize}
Then every group scheme in $\catc$ that admits a filtration using only copies of the
group schemes $G_i[\pi_i]$ for $i \in I$ is of the following form:
$$ 
\bigoplus_{s_i=1}^{r_i} G_i[\pi_i^{n_{s_i}}]  \quad \text{where} \quad n_{s_i} \in \N .
$$
\end{theorem}
\begin{proof}
This follows immediately from \cite[Corollary 8.2]{Schoof:2005} and the condition that
the groups of extensions $\extt{\catc}{G_i[\pi_i]}{G_j[\pi_j]}$ and $\extt{\catc}{G_j[\pi_j]}{G_i[\pi_i]}$ are trivial.
\end{proof}

This theorem allows us to prove:

\begin{corollary} \label{j049-cor:twodivfiltered}
Any group scheme in the category $\cate$ that can be filtered by copies of $\calE[\pi]$ and $\calE[\pibar]$
is isomorphic to a group scheme of the form 
$$
\bigoplus_{s=1}^{r_1} \calE[\pi^{n_{s}}] \times \bigoplus_{t=1}^{r_2} \calE[\pibar^{n_{t}}] .
$$
\end{corollary}
\begin{proof}
We verify that the $2$-divisible groups 
$G_1 = \calE[\pi^\infty]$ and $G_2 = \calE[\pibar^\infty]$
satisfy the conditions of Theorem \ref{j049-thm:grpschm_filtration}
where we take the category $\catc$ to be the category $\cate$.

We verify the first condition: 
The Galois representation $\rho_{E,\good}$ is absolutely irreducible since it is an induction
of a Hecke character from the quadratic imaginary field $\Q(\sqrt{-7})$ to $\Q$.
By Schur's Lemma, we obtain that the endomorphism ring 
$\End_{G_\Q}(T_2(E) \otimes \Q_2)$ is isomorphic to $\Q_2$.
Therefore, $\End_{G_\Q}(T_2(E)) \simeq \Z_2$.
By Tate's Theorem \cite[Corollary of Theorem 4]{Tate:1967},
we have that $\End_{G_\Q}(T_2(E)) \simeq \End(\calE[2^\infty])$.
Since the prime $2$ splits in $K/\Q$, it then follows
that also $\End(\calE[\pi^\infty]) \simeq \Z_2$ and $\End(\calE[\pibar^\infty]) \simeq \Z_2$.

The third condition holds by Lemma \ref{j049-lem:ext_gpi_by_gpi}:
The group $\extt{\cate}{\calE[\pi]}{\calE[\pi]}$ is $1$-dimensional over $\F_2$.
The group of homomorphisms
$\Hom_{O_K[\frac{1}{7}]}(\calE[\pi],\calE[\pi])$
is also $1$-dimensional and we see that $\delta$ is an isomorphism.
The same holds for the group scheme $\calE[\pibar]$ instead of $\calE[\pi]$.
\end{proof}

However, we cannot simply apply Theorem \ref{cat-thm:abvar_torsionfilter} to
deduce that the $2^n$-torsion subgroup scheme of $\calA$ can only be filtered
by the group schemes $\calE[\pi]$ and $\calE[\pibar]$.
Lemma \ref{j049-lem:ext_mu2_Z2_Gpi} shows that the first condition of
Theorem \ref{cat-thm:abvar_torsionfilter} does not hold.
Nevertheless, we have the following result:

\begin{lemma} \label{j049-lem:isogeny}
The abelian varieties $A$ and $E^{\dim(A)}$ are isogenous over $K$.
\end{lemma}
\begin{proof}
The group schemes $M(n)$ and $\calA[2^n]/\gsc''(n)$ have bounded rank
using the argument in the proof of \cite[Proposition 2.6.3]{Verhoek:2009}. 

Since $\calA[2^{n-1}] \subset \calA[2^n]$, we also have that 
$M(n-1)$ is a subgroup scheme of $\calA[2^n]$.
So there exists an $m \in \N$ such that for all $n > m$, we can find a filtration of $\calA[2^n]$
such that $M(n) \simeq M(m)$. 
The group scheme $\calA[2^n]/M(m)$ does not contain a copy of $\mu_2$ in its filtration because
the rank of $M(n)$ is equal to the rank of $M(m)$.
If $n \geq m$, let $k(n)$ be the smallest positive integer such that 
$[2^{k(n)}] (A[2^n] )$ is a closed flat subgroup scheme of $\gsc''(n)$.
Let $q(n)$ be the quotient map $\calA[2^n] \rightarrow \calA[2^n]/M(m)$.
We have morphisms of group schemes
$$
\phi_n := q(n) \circ  [2^k(n)]  : \calA[2^n] \longrightarrow \gsc''(n) /M(m) \simeq \calE[2^n]^{\dim(A)}.
$$
The last isomorphism follows from Corollary \ref{j049-cor:twodivfiltered}.
The kernel and cokernel of the morphisms $\phi_n$ are bounded
because the numbers $k(n)$ and the rank of $M(m)$ are bounded.  

Consider the following diagram for $n \geq m$:
$$
\begin{xy}
\xymatrix{
\calA[2^n] \ar[d]^{[2]} \ar[rr]^{ \phi_n} & & \calE[2^n]^{\dim(A)} \ar[d]^{[2]}  \\
\calA[2^{n-1}]   \ar[rr]^{ \phi_{n-1} } &  & \calE[2^{n-1}]^{\dim(A)}  .
}
\end{xy} 
$$
The diagram commutes if and only if $k(n) = k(n-1)$.
There exists an infinite subset $I \subset \N$ such that the $(\phi_i)_{i \in I}$ are compatible,
because the numbers $k(n)$ are bounded and $k(n+i) \leq k(n)+i$ for all
positive integers $i$.
In other words, there exists a cofinal compatible system of the morphisms $\phi_n$.

Now let $\calA[2^\infty]$ denote the $2$-divisible group associated to $A$.
The cofinal compatible system of morphisms $\phi_n$ gives rise
to a morphism of $2$-divisible groups as follows.
For each $n$ there exists an $n' > n$ such that $n' \in I$ and
we let 
$$
f_n :=  \phi_{n'}| \calA[2^n] : \calA[2^n] \longrightarrow \calE[2^n]^{\dim(A)} .
$$
The family $(f_n)_{n \in \N}$ is then a morphism
of $2$-divisible groups $\calA[2^{\infty}] \rightarrow \calE[2^{\infty}]^{\dim(A)}$
with bounded kernel and cokernel.
By \cite{Faltings:1983}, it follows that $A$ is isogenous to 
$E^{\dim(A)}$ over $K$.
\end{proof}

In the next lemma we switch from $2$-divisible groups to $2$-adic Tate modules.
Recall that for an abelian variety $A$ over $\Q$, 
we write $\rho_{A,2} : G_\Q \rightarrow \Aut(T_2(A) ) \otimes \Q_2$
for the representation that describes the action of $G_\Q$ on the $2$-adic Tate module of $A$.
Also recall that $B$ is the $2$-dimensional abelian variety over $\Q$ attached to the newform $49B$.

\begin{lemma} \label{j049-lem:tatemod_descent}
Every representation $\rho : G_\Q \rightarrow \Gl_{2g}(\Q_2)$
such that
$$
\rho | \Gal(\Alg{\Q}/\Q(\zeta_7)) \simeq (\rho_{E,2}| \Gal(\Alg{\Q}/\Q(\zeta_7)))^g \quad \text{for some} \quad g \in \N,
$$
is a direct sum of copies of $\rho_{E,2}$ and $\rho_{B,2}$.
\end{lemma}
\begin{proof}
Let $\Gamma_1 = \Gal(\Alg{\Q}/\Q)$, 
$\Gamma_2 =  \Gal(\Alg{\Q}/\Q(\sqrt{-7}))$
and $\Gamma_3 =  \Gal(\Alg{\Q}/\Q(\zeta_7))$,
so that $\Gamma_1 \supset \Gamma_2 \supset \Gamma_3$.
We begin by remarking that $\rho$ is semi-simple:
Let $V$ be the vector space on which $\Gamma_1$ acts through $\rho$,
and let $W$ be a subspace stable under $\Gamma_1$.  
The representation $\rho | \Gamma_3$ is semi-simple.
The projection formula
$$
\frac{1}{[\Gamma_1:\Gamma_3]} \sum_{\sigma \in \Gamma_1/\Gamma_3} \rho(\tilde{\sigma}) \circ \Pi \circ \rho(\tilde{\sigma})^{-1} ,
$$
where $\tilde{\sigma}$ is a lift of $\sigma$ and where $\Pi$ is a $\Gamma_3$-linear projection of $V$ onto $W$,
shows that the complement of $W$ in $V$ is also $\Gamma_1$-stable.
So also $\rho$ is semi-simple.
This implies that we may reduce to the case that $\rho$ is irreducible.

Since $E$ has CM over $\Q(\sqrt{-7})$ (so in particular it has CM over $\Q(\zeta_7)$), 
the image of the representation $\rho_{E,2}|\Gamma_3$ is
$$
\left(\begin{array}{cc}\psi & 0 \\0 & \psi(\tilde{\sigma} \text{ . } \tilde{\sigma}^{-1})  \end{array}\right) | \Gamma_3
$$
where $\psi : \Gamma_2 \rightarrow \Q_2^*$ is a certain character and 
where $\tilde{\sigma}$ lifts $1 \not= \sigma \in \Gal(\Q(\sqrt{-7})/\Q)$.

By assumption, $\Hom_{\Gamma_3}(\rho|\Gamma_3, \rho_{E,2}|\Gamma_3)$ is not trivial.
In particular, $\Hom_{\Gamma_3}(\rho|\Gamma_3, \psi | \Gamma_3)$ is not trivial.
Then by Frobenius reciprocity we also have that
$\Hom_{\Gamma_1} \left( \rho , \Ind_{\Gamma_3}^{\Gamma_1}( \psi |\Gamma_3) \right)$ is non-trivial.
Therefore we study the irreducible representations of
$\Ind_{\Gamma_3}^{\Gamma_1}( \psi |\Gamma_3)$.
First we note that 
$$
\Ind_{\Gamma_3}^{\Gamma_1}( \psi |\Gamma_3) \simeq 
\Ind_{\Gamma_2}^{\Gamma_1} \left( \Ind_{\Gamma_3}^{\Gamma_2}( \psi |\Gamma_3) \right) .
$$
The representation $\Ind_{\Gamma_3}^{\Gamma_2}( \psi |\Gamma_3)$
is isomorphic to $\Ind_{\Gamma_3}^{\Gamma_2}( \Q_2^{\text{triv}} ) \otimes \psi |\Gamma_2$,
where $\Q_2^{\text{triv}}$ denotes the trivial $1$-dimensional representation.
The representation $\Ind_{\Gamma_3}^{\Gamma_2}( \Q_2^{\text{triv}} )$ is
isomorphic to $\Q_2^{\text{triv}} \oplus \chi$, where $\chi$ is the $2$-dimensional representation
$\chi : \Gamma_2/\Gamma_3 \rightarrow \Gl_2(\Q_2)$ that
corresponds to the non-trivial character of 
$\Gal(\Q(\zeta_7)^+/\Q)$ to $\F_4^*$ that we saw before.
To summarize, we have isomorphisms
\begin{align*}
\Ind_{\Gamma_3}^{\Gamma_1}( \psi |\Gamma_3) \simeq & 
\Ind_{\Gamma_2}^{\Gamma_1} \left(  ( \Q_2^{\text{triv}} \oplus \chi ) \otimes \psi |\Gamma_2  \right) \simeq
\Ind_{\Gamma_2}^{\Gamma_1} \left(  \psi |\Gamma_2 \oplus (\chi \otimes \psi |\Gamma_2 ) \right) \\
\simeq
& \Ind_{\Gamma_2}^{\Gamma_1}(  \psi |\Gamma_2) \oplus \Ind_{\Gamma_2}^{\Gamma_1}(\chi \otimes \psi |\Gamma_2  )
\simeq \rho_{E,2}  \oplus \Ind_{\Gamma_2}^{\Gamma_1}(\chi \otimes \psi |\Gamma_2 )  .
\end{align*}
The representation $\Ind_{\Gamma_2}^{\Gamma_1}(\chi \otimes \psi |\Gamma_2)$ is irreducible
by Mackey's criterion, see for example \cite[Proposition 23, p. 59]{Serre:1977}. Although we are dealing
here with infinite Galois groups and Mackey's criterion is only stated for finite groups,
we can just restrict the representations of the infinite Galois groups to
finite Galois groups and the induction restricted to these finite groups
is an irreducible representation. Hence $\Ind_{\Gamma_2}^{\Gamma_1}(\chi \otimes \psi |\Gamma_2)$ is also
irreducible.
Furthermore, the representation $\Ind_{\Gamma_2}^{\Gamma_1}(\chi \otimes \psi |\Gamma_2)$
is isomorphic to $\rho_{B,2}$ by the following argument.
We know that the group 
$\Hom_{\Gamma_3}(\rho_{B,2}|\Gamma_3 , (\rho_{E,2} \oplus \rho_{E,2}) | \Gamma_3  )$ is non-trivial,
hence
$\Hom_{\Gamma_1} \left( \rho_{B,2}, \Ind_{\Gamma_3}^{\Gamma_1} (\psi | \Gamma_3 ) \right)$ is non-trivial.
We also know that the abelian variety $B$ is not isomorphic to $E$.
We conclude that $\rho_{B,2} \simeq \Ind_{\Gamma_2}^{\Gamma_1}(\chi \otimes \psi |\Gamma_2 )$.
\end{proof}

The above lemma, together with our previous results, 
describes the $2$-adic Tate module of the abelian variety $A$ tensored with $\Q_2$.
It allows us to prove Theorem \ref{j049-thm:main}:

\begin{proof}[Proof of Theorem \ref{j049-thm:main}]
Let $A$ be an abelian variety over $\Q$
with good reduction outside $7$ that obtains semi-stable reduction
over an at most tamely ramified extension at $7$ over $\Q$.
It follows from Lemma \ref{j049-lem:isogeny} that
$$
\rho_{A,2} | \Gal(\Alg{\Q}/\Q(\zeta_7)) \simeq (\rho_{E,2}| \Gal(\Alg{\Q}/\Q(\zeta_7)))^{\dim(A)} 
$$
holds.
Therefore, we may apply Lemma \ref{j049-lem:tatemod_descent} 
and deduce that 
$$
\rho_{A,2} \simeq \rho_{E,2}^a \oplus \rho_{B,2}^b ,
$$
for $a$ and $b$ integers such that $a+2b=\dim(A)$.
By Falting's Theorem \cite{Faltings:1983}, we have
$$
\Hom(\rho_{A,2}  , \rho_{E,2}^a \oplus \rho_{B,2}^b )
\simeq \Hom(A,E^a \times B^b) \otimes \Q_2 .
$$
It follows that $A$ is isogenous to the abelian variety $E^a \times B^b$.
\end{proof}

\section{Conductors}

In this section we prove Theorem \ref{thm:conductorj027}
and \ref{thm:conductorj049}.
First, we briefly recall the definition of the conductor of an abelian variety.
Let $\locfield$ be a local field with uniformizer $\pi$ and ring of integers $O_\locfield$.
Furthermore, let $A$ be an abelian variety over $\locfield$, $\calA$ the N\'eron model of $A$
and $\calA^{0}_{\pi}$ be the connected component of the N\'eron model in the special fiber.
Consider the exact sequence of commutative algebraic groups over $O_\locfield/\pi O_\locfield$
$$
0 \longrightarrow T_{U_{A}} \longrightarrow \calA^{0}_{\pi} \longrightarrow B_A \longrightarrow 0 , 
$$
where $B$ is an abelian variety and $T_{U_A}$ an affine connected smooth scheme.
The scheme $T_{U_A}$ sits inside an exact sequence
$$
0 \longrightarrow U_A \longrightarrow T_{U_{A}} \longrightarrow T_A \longrightarrow 0
$$
with $T_A$ a torus and $U_A$ a unipotent group. Let $u_A=\dim(U_A)$ and $t_A=\dim(T_A)$.

\begin{definition}
Let $\good$ be a rational prime different from the residue characteristic of $\locfield$
and denote the lower ramification groups $({I_{\locfield(A[\good])/\locfield}})_{i}$ by $G_{i}$.
Then \emph{Serre's measure of wild ramification of $A$} is
$$
\delta_A := \sum_{i=1}^{\infty} \frac{ \# G_{i} }{ \# G_{0} } 
\dim_{O_\locfield/\pi O_\locfield}(A[\good](\Alg{\locfield}) / A[\good](\Alg{\locfield})^{G_{i} } ) .
$$
\end{definition}

We remark that $\delta_A$ does not depend on $\good$.

\begin{definition}
The \emph{exponent of the conductor of $A$ over a local field $\locfield$} 
is defined to be $c(\pi) = 2u_A+ t_A + \delta_A$.
The \emph{conductor of an abelian variety $A$ over a number field $K$} is defined to be
$c(A) = \prod_{\frakbad} \frakbad^{c(\frakbad)}$, where the product is taken 
over the prime ideals in $O_K$ and
$c(\frakbad)$ is the exponent of the conductor of $A$ over $K_{\frakbad}$.
\end{definition}

Note that $c(\frakbad)=0$ for almost all primes $\frakbad$, so the product is well-defined. 
To prove Theorems \ref{thm:conductorj027} we need a small group theoretic lemma:

\begin{lemma} \label{j027-lem:3actson2}
Let $G$ be a $3$-group acting non-trivially on a vector space $V \simeq \F_{2}^{2d}$.
Then $\dim_{\F_{2}} (V^{G})$ is even and strictly less than $2d$.
\end{lemma}
\begin{proof}
We have the usual orbit formula
$\# V^{G} + \sum_{v \in V/ \sim} Gv = 2^{2d}$.
For $v \in V$ and $v \notin V^{G}$, the length of the orbit $Gv$ is divisible by $3$.
Therefore, considering the above equality modulo $3$, we obtain that
$\# V^{G} \equiv 2^{2d} \equiv 1 \pmod{3}$.

Suppose that $\dim V^{G} =2d-x$ with $x \in \N$.
Then $2^{2d-x} \equiv 2^{x} \equiv 1 \pmod{3}$, hence $x$ cannot be odd.
By assumption, $G$ acts non-trivially on $V$ and so $x \not= 0$.
Hence $\dim V^{G}$ is even, as we wanted to show.
\end{proof}

\begin{proof}[Proof of Theorem \ref{thm:conductorj027}]
Let $A$ be an abelian variety over $\Q$ of conductor $27$.
Then using the notation from above,
the exponent of the conductor of $E$ at $3$ is equal to $3=2u_A+t_A+\delta_A$.
For the elliptic curve $E$,
we find that $u_E=1, t_E=0$ and $\delta_E=1$.

If $u_A=0$, then $A$ is semi-stable over $\Q$.
By Theorem \ref{j027-thm:main}, $A$ is then isogenous to a product of $E$.
This contradicts the fact that $E$ is not semi-stable.
We conclude that $u_A$ must equal $1$.
 
Next suppose that $t_A=1$ and hence that $\delta_A=0$.
But if $\delta_A=0$, then $\rho_{A,\good}$ is at most tamely ramified at $3$ for some prime $\good$.
In particular, $A$ is ramified of level $\frac{1}{2}$ at $3$,
and it follows from Theorem \ref{j027-thm:main} that $A$ is isogenous to a product
of $J_{0}(27)$. This implies that $\rho_{A,\good}$ cannot be at most tamely ramified at $3$,
contradicting our assumption.
We conclude that we must have $u_A=1,t_A=0$ and $\delta_A=1$.

If we denote the lower ramification groups $I_{\Q(A[2]) / \Q}(3)_{i}$ by $G_{i}$,
then:
$$
\delta_A = \sum_{i=1}^{\infty} \frac{ \# G_{i} }{ \# G_{0} } \dim_{\F_{2}}(A[2](\Alg{\Q}) / A[2](\Alg{\Q})^{G_{i} } ) .
$$
Apply Lemma \ref{j027-lem:3actson2}
by taking $G=G_{i}$ and $V = A[2](\Alg{\Q}) \simeq (\Z/2\Z)^{2d}$. 
It follows that if $\# G_{i} \not= 1$, then 
$\dim_{\F_{2}}(A[2](\Alg{\Q}) / A[2](\Alg{\Q})^{ G_{i}  }) \geq 2$. 
This in turn implies
$$
\sum_{\substack{ G_{i} \not= 1 , \\ i > 0 }} \frac{ \# G_{i} }{ \# G_{0} } \leq \frac{1}{2} ,
$$
hence $A$ is ramified of level $\frac{1}{2}$ at $3$.
By Theorem \ref{j027-thm:main} $A$ is isogenous to a product of $J_{0}(27)$.
Since $A$ has conductor $27$, $A$ must be isogenous to $J_{0}(27)$.
\end{proof}

\begin{proof}[Proof of Theorem \ref{thm:conductorj049}]
Let $\calA^0_7$ be the connected component of the special fiber at $7$
of the N\'eron model $\calA$ of $A$.
Write the exponent of the conductor at $7$
as the sum $2=2 u + t + \delta$,
where $u, t$ and $\delta$ are as before.
Since $A$ is not semi-stable by \cite[Theorem 1.1]{Schoof:2005}, we have
that $u > 0$.
Hence, $t=\delta=0$. 
Therefore, we may apply Theorem \ref{j049-thm:main}
and conclude that $A$ is isogenous to $J_0(49)$.
\end{proof}

Under the assumption that the $L$-series of abelian varieties over $\Q$ have 
an analytic continuation and satisfy a functional equation
(which conjecturally holds for all abelian varieties over $\Q$),
Theorem \ref{thm:conductorj027} also follows from a result of Mestre \cite[Proposition, p.21]{Mestre:1986}.
Mestre shows that the conductor $N$ of an abelian variety $A$ over $\Q$
of dimension $g$ that satisfies these hypotheses,
respects the lower bound $N > 10^{g}$.
This implies that the abelian varieties of conductor $27$ are elliptic curves over,
in which case Theorem \ref{thm:conductorj027} is clear.

\section{Appendix} \label{sec:appendix}

\noindent
{\small
Calculations for $\Q(\calJ_0(32)[3])$:}

{\tiny
\begin{verbatim}
E := EllipticCurve([0,0,0,-1,0]);
_<x> := PolynomialRing(Rationals());
f := DivisionPolynomial(E,3);
K := SplittingField(f); 
_<y> := PolynomialRing(K);
xs := [r[1] : r in Roots(f,K) | IsIrreducible(y^2-r[1]^3 + r[1])]; 
L := ext<K | y^2-xs[1]^3+xs[1]>;
Labs := AbsoluteField(L);
G := GaloisGroup(Labs);
A := Group< s,t  | s^8 = t^2 = 1, s*t = t*s^3>;
A1 := PCGroup(A);
G1 := PCGroup(G);
IsIsomorphic(G1,A1);
a := PrimitiveElement(Labs);
H := ext<RationalField()| MinimalPolynomial(3*a) >;
O_H := MaximalOrder(H);
p3 := Decomposition(O_H, 3)[1][1];
p2 := Decomposition(O_H, 2)[1][1];
[ <i,j,Degree(RayClassField(p2^i * p3^j))> : i in [1..8] , j in [0,1] ];
[ <i,j,Degree(RayClassField(p3^i * p2^j))> : i in [1..5] , j in [0,1] ];
\end{verbatim}
}

\noindent
{\small
We calculate the extension $\Q(\calJ_0(27)[4])$:}

{\tiny
\begin{verbatim}
E := EllipticCurve([0,0,1,0,0]);
_<x> := PolynomialRing(Rationals());
f := DivisionPolynomial(E,4);
K := SplittingField(f);
_<y> := PolynomialRing(K);
xs := [r[1] : r in Roots(f,K) | IsIrreducible(y^2+y-r[1]^3)];
L := ext<K | y^2+y-xs[1]^3>;
Labs := AbsoluteField(L);
a := PrimitiveElement(Labs);
H := ext<RationalField()| MinimalPolynomial(2*a) >;
O_H := MaximalOrder(H);
Factorization(Discriminant(O_H));
\end{verbatim}
}

\noindent
{\small
We calculate the ramification groups and ray class field for $\Q(\calJ_0(27)[4])$:}
{\tiny
\begin{verbatim}
p2 := Decomposition(O_H, 2)[1][1];
p3 := Decomposition(O_H, 3)[1][1];
[ Order(RamificationGroup(p2,i)) : i in [0..5] ];
[<i,j,Degree(RayClassField(p2^i * p3^j))> : 
i in [1..10] , j in [0..1]];
[<i,j,Degree(RayClassField(p3^i * p2^j))> : 
i in [1..6] , j in [0..1]];
\end{verbatim}
}

\noindent
{\small
We check that the ray class field of conductor $\pi_7$ of the field $\Q(\zeta_7)$
is trivial:}
{\tiny
\begin{verbatim}
K := CyclotomicField(7);
p7 := Decomposition(OK, 7)[1][1];
RayClassField(p7);
\end{verbatim}
}

\noindent
{\small
Let $R$ be the ray class field of conductor $\pi_7 \pi^2$ of the field $\Q(\zeta_7)$.
We check that the prime $\pibar$ does not split in $R/\Q$:}

{\tiny
\begin{verbatim}
K := CyclotomicField(7);
p21 := Decomposition(OK, 2)[1][1];
p22 := Decomposition(OK, 2)[2][1];
p7 := Decomposition(OK, 7)[1][1];
R := AbsoluteField(NumberField(RayClassField(p7 * p21^2)));
OR := MaximalOrder(R);
Decomposition(OR, 2);
\end{verbatim}
}

\bibliography{index}
\bibliographystyle{alpha}
\end{document}